\documentclass[12pt]{amsart}
\usepackage{verbatim}
\usepackage{amssymb}
\usepackage{amsmath}
\usepackage[usenames]{color}
\usepackage{graphicx}
\usepackage{hyperref}
  \scrollmode

\newcommand{\R}{{\mathbb R}}
\newcommand{\N}{{\mathbb N}}

\newcommand{\EE}{{\mathbb E}}
\newcommand{\EEE}{\operatorname{E}}
\newcommand{\PP}{{\mathbb P}}
\newcommand{\E}{{\mathbb E}}
\newcommand{\A}{\mathcal A}
\newcommand{\cF}{\mathcal F}
\newcommand{\ccF}{\mathcal G}
\newcommand{\Hc}{\mathcal H}
\newcommand{\V}{\mathcal V}
\newcommand{\cY}{\mathcal Y}

\newcommand{\ccFt}{\widetilde {\mathcal G}}
\newcommand{\cP}{\mathcal P}

\newcommand{\Xh}{\widehat X}
\newcommand{\Sh}{\widehat S}

\newcommand{\Adet}{\mathcal{A}^{\text{det}}}
\newcommand{\Aran}{\mathcal{A}^{\text{ran}}}
\newcommand{\edet}{e^{\text{det}}}
\newcommand{\eran}{e^{\text{ran}}}
\newcommand{\cost}{\operatorname{cost}}

\newcommand{\eq}{{\mathcal{E}}}

\newcommand{\1}{1}
\newcommand{\sSDE}{S^\text{sde}}
\newcommand{\ff}{g}

\newcommand{\eps}{\varepsilon}
\newcommand{\thet}{\theta}
\newcommand{\hehe}{h}
\newcommand{\xl}{x_\delta}
\newcommand{\vl}{\zeta_u}

\theoremstyle{plain}
\newtheorem{theorem}{Theorem}
\newtheorem{prop}{Proposition}
\newtheorem{lemma}{Lemma}
\newtheorem{cor}{Corollary}

\theoremstyle{definition}
\newtheorem{rem}{Remark}

\begin{document}

\title[]{On hard  quadrature problems \\ for marginal distributions of SDEs \\
with bounded smooth coefficients}

\author[M\"uller-Gronbach]
{Thomas M\"uller-Gronbach}
\address{
Fakult\"at f\"ur Informatik und Mathematik\\
Universit\"at Passau\\
Innstrasse 33 \\
94032 Passau\\
Germany} \email{thomas.mueller-gronbach@uni-passau.de}

\author[Yaroslavtseva]
{Larisa Yaroslavtseva}
\address{
Fakult\"at f\"ur Informatik und Mathematik\\
Universit\"at Passau\\
Innstrasse 33 \\
94032 Passau\\
Germany} \email{larisa.yaroslavtseva@uni-passau.de}

\begin{abstract}
In recent work of Hairer, Hutzenthaler and Jentzen, see~\cite{hhj12}, a stochastic differential equation (SDE) with infinitely often differentiable and
bounded coefficients
 was
constructed
such that the Monte Carlo Euler method for
approximation of the expected value
of the first component
of the solution at the final time converges but fails to achieve a mean square error of
    a polynomial rate.
In the present paper we show that this type of bad performance for
quadrature of SDEs with infinitely often differentiable and bounded
coefficients is not a shortcoming of the Euler scheme in particular
but can be observed in a worst case sense for every approximation
method that is based on finitely many function values of the
coefficients of the SDE. Even worse we show that for any sequence of
Monte Carlo methods based on finitely many sequential evaluations of
the coefficients and all their partial derivatives and for every
arbitrarily slow convergence speed there exists a sequence of SDEs
with infinitely often differentiable and bounded by one coefficients
such that the first order derivatives of all diffusion coefficients
are bounded by one as well and the first order derivatives of all
drift coefficients are uniformly dominated by a single real-valued
function and such that the corresponding sequence of mean absolute
errors for approximation of the expected value of the first
component of the solution at the final time can not converge to zero
faster than the given speed.
\end{abstract}

\maketitle

\section{Introduction}
Let $d,m\in \N$ and consider a $d$-dimensional system of autonomous SDEs
\begin{equation}\label{sde}
\begin{aligned}
\phantom{\quad\qquad t\in [0,1],}
dX^{a,b}(t) & = a(X^{a,b}(t)) \, dt + b(X^{a,b}(t)) \, dW(t), \quad t\in [0,1],\\
X^{a,b}(0) & = 0,
\end{aligned}
\end{equation}
with an $m$-dimensional Brownian
motion $W$ and infinitely often differentiable, bounded coefficients
$a\colon\R^d \to \R^d$ and $b\colon\R^d\to \R^{d\times m}$. In
particular, there exists a unique
strong solution $X^{a,b}=(X_1^{a,b},\dots,X_d^{a,b})$ of \eqref{sde}, see, e.g.,
Theorem 3.1.1 in \cite{PrevotRoeckner2007},
 and we have $\EEE[|X^{a,b}(1)|^p] <
\infty$ for every $p \geq 1$. Let $f\colon \R^d\to\R$ be a measurable function that
satisfies a polynomial growth condition.
   We study the computational task
to approximate the quantity
\[
S(a,b,f)=\EEE \bigl[f (X^{a,b}(1))\bigr]
\]
by means of an algorithm that uses function values of $a$, $b$
and $f$ and, eventually, their partial derivatives $D^\alpha a$, $D^\alpha b$, $D^\alpha f$ at finitely many points in $\R^d$.

A classical method of this type is given by the Monte Carlo
quadrature rule $\widehat S_n^\text{E}$ based on $n$ repetitions of
the Euler scheme $(\Xh^{a,b}(\ell/n))_{\ell=0,\dots,n}$  with time
step $1/n$, i.e.,
\[
\widehat S_n^\text{E} (a,b,f) = \frac{1}{n} \sum_{i=1}^n f(Y^{a,b}_i),
\]
where $Y^{a,b}_1,\dots,Y^{a,b}_n$ are independent and identically distributed as  $\Xh^{a,b}(1)$ and the scheme $\Xh^{a,b}$ is recursively defined by $\Xh^{a,b}(0) = 0$ and
\[
\Xh^{a,b}(\tfrac{\ell}{n}) = \Xh^{a,b}(\tfrac{\ell-1}{n}) + a(\Xh^{a,b}(\tfrac{\ell-1}{n}))\cdot \tfrac{1}{n} + b(\Xh^{a,b}(\tfrac{\ell-1}{n}))\cdot (W(\tfrac{\ell}{n})-W(\tfrac{\ell-1}{n}))
\]
for $\ell = 1,\dots,n$. Then it is easy to see that
\[
\EE\bigl[|S(a,b,f)- \widehat S_n^\text{E}(a,b,f)|^2\bigr] \le
c_1\cdot \|f\|_{\text{Lip}}^2\cdot \exp\bigl(c_2\cdot \max_{\alpha\in \N_0^d, |\alpha|_1 = 1}(\|D^\alpha a\|^2_\infty+\|D^\alpha b\|^2_\infty)\bigr)\cdot \frac{1}{n},
\]
where
 $\|f\|_{\text{Lip}}$ denotes the Lipschitz seminorm of $f$, $|\alpha|_1 = \sum_{i=1}^d |\alpha_i|$  and
 $c_1$ and $c_2$ are positive reals, which only depend on the dimensions $d$ and $m$.
Thus, if the first order partial derivatives of the coefficients $a$ and $b$ are
   also bounded and the integrand $f$ is Lipschitz continuous then
the sequence of Monte Carlo Euler approximations $\widehat S_n^\text{E}(a,b,f)$
achieves a polynomial rate of root mean square error convergence of order $1/4$ in terms of the total number $2n(n-1)+n$ of evaluations of the coefficients $a$ and $b$ and the integrand $f$.

On the other hand, Hairer, Hutzenthaler and Jentzen  have presented in \cite{hhj12} an equation  \eqref{sde} with $d=4$, $m=1$ and infinitely often differentiable, bounded coefficients $a,b$ such that
for the integrand $f(x_1,\dots,x_d)=x_1$ and every $\kappa \in (0,\infty)$,
\begin{equation}\label{zz1}
\lim_{n\to\infty} n^\kappa \cdot \EE\bigl[|S(a,b,f)- \widehat S_n^\text{E}(a,b,f)|^2\bigr] = \infty.
\end{equation}
Hence the sequence of  Monte Carlo Euler  approximations $\widehat S_n^\text{E}(a,b,f)$
 might not achieve a polynomial rate of root mean square error convergence in terms of the number of evaluations of the coefficients $a$ and $b$ and the integrand $f$ if the first order partial derivatives of the coefficients are not bounded as well.

It seems natural to ask, whether the latter result demonstrates only a particular fallacy of the Monte Carlo Euler method and a polynomial rate of convergence could always be achieved for equations \eqref{sde} with infinitely often differentiable and bounded coefficients $a$, $b$ if only a more advanced approximation scheme than the Euler scheme would be employed. In fact, there is a variety of strong approximation schemes available in the literature that have been constructed to cope with non-Lipschitz
   continuous
coefficients and have been shown to achieve a polynomial rate of convergence, in terms of the number of
   time
steps, for   suitable classes of such equations. See, e.g., \cite{h96,hms02,Schurz2006,MaoSzpruch2013Rate,HutzenthalerJentzenKloeden2012,
WangGan2013,Sabanis2013ECP,Sabanis2013Arxiv,Beynetal2014,TretyakovZhang2013,
Beynetal2015,KumarSabanis2016}
for equations with
   globally
 monotone coefficients and see, e.g.,
 \cite{BerkaouiBossyDiop2008,GyoengyRasonyi2011,DereichNeuenkirchSzpruch2012,
Alfonsi2013,NeuenkirchSzpruch2014,HutzenthalerJentzen2014,
HutzenthalerJentzenNoll2014CIR,ChassagneuxJacquierMihaylov2014}  for equations with possibly non-monotone coefficients.

However, the following result, Theorem \ref{ti1}, which is a straightforward consequence of Corollary \ref{cor1}
in Section~\ref{sub4-1}, shows that for $d\ge 4$ the pessimistic alternative is true in a worst case sense with respect to the coefficients $a$ and $b$.
For every sequence of Monte Carlo methods based on some kind of It\^o-Taylor scheme there exists a sequence of equations \eqref{sde} with
infinitely often differentiable and bounded by one coefficients  such that the resulting sequence of mean absolute errors
for approximating the expected value of the first component of the solution at the  final time does not converge to zero with a polynomial rate.

To state this finding in a more formal way
let
\[
I=\bigcup_{k=1}^\infty \{0,1\}^k
\]
denote  the set of all finite sequences of zeros and ones, put
$W_0(t)=t$ for $t\in[0,1]$, and for $\beta\in I$, $n\in\N$ and
$\ell\in\{1,\dots,n\}$ let
\[
J_{n,\ell}^\beta = \int_{\tfrac{\ell-1}{n}}^{\tfrac{\ell}{n}}\dots \int_{\tfrac{\ell-1}{n}}^{u_2} 1\,dW_{\beta_1}(u_1)\dots dW_{\beta_\ell}(u_\ell)
\]
denote the corresponding iterated It\^o-integral over the time interval $[\tfrac{\ell-1}{n},\tfrac{\ell}{n}]$.

\begin{theorem}\label{ti1}
Let $d=4, m=1$
and let $\varphi\colon (\R^4\times \R^{4})^{\N_0^4}\times \R^I\to \R^4$ be a measurable mapping. For all infinitely often differentiable functions $a\colon \R^4\to \R^4$ and $b\colon \R^4\to \R^{4}$ and every $n\in\N$ define the scheme $\bigl(\Xh_n^{a,b}(\tfrac{\ell}{n})=(\Xh_{n,1}^{a,b},\dots,\Xh_{n,4}^{a,b})(\tfrac{\ell}{n}) \bigr)_{\ell=0,\dots,n}$ by $\Xh_n^{a,b}(0) = 0$ and
\[
\Xh_n^{a,b}(\tfrac{\ell}{n}) = \Xh_n^{a,b}(\tfrac{\ell-1}{n}) + \varphi\bigl((D^\alpha a(\Xh_n^{a,b}(\tfrac{\ell-1}{n}) ),D^\alpha b(\Xh_n^{a,b}(\tfrac{\ell-1}{n}) ))_{\alpha\in \N_0^4}, (J^\beta_{n,\ell})_{\beta\in I}\bigr),
\]
and let $Y^{a,b}_{n,1},\dots,Y^{a,b}_{n,n}$ be independent and identically distributed as  $\Xh_{n,1}^{a,b}(1)$.

Then there exist sequences $(a_n)_{n\in\N}$ and $(b_n)_{n\in\N}$ of infinitely often differentiable functions $a_n\colon \R^4\to \R^4$ and $b_n\colon \R^4\to \R^4$ with $\|a_n\|_\infty\le 1$ and $ \|b_n\|_\infty\le 1$ such that for every $\kappa \in (0,\infty)$,
\[
\lim_{n\to\infty} n^\kappa \cdot \EE \Bigl[\bigl| \EEE [X^{a_n,b_n}_1(1)] - \frac{1}{n}\sum_{k=1}^n Y^{a_n,b_n}_{n,k}\bigr|\Bigr] = \infty.
\]
\end{theorem}

The latter result neither covers Multilevel Monte Carlo schemes
   nor the case of a non-uniform discretization of time.
Moreover, one might argue that a result like Theorem~\ref{ti1} is
not surprising since the order one partial derivatives of the chosen
coefficients $a_n$ and $b_n$ in the theorem are not required to
simultaneously satisfy some
 kind of growth condition.
 However, from Corollary \ref{cor1} in Section~\ref{sub4-1} we even obtain that
 the coefficients $a_n$ and $b_n$ in Theorem~\ref{ti1} can be chosen in such a way that the order one partial derivatives of $b_n$ are bounded by one as well and the order one partial derivatives of $a_n$ are dominated by the function $x\mapsto 1 + \exp(|x|^3)$, and
that furthermore the statement of the theorem  extends to any sequence of Monte Carlo  methods based on sequential evaluation of the coefficients $a$ and $b$ and all their partial derivatives $D^\alpha a$, $D^\alpha b$  at finitely many points in $\R^d$. More formally, we have the following theorem as an immediate consequence of Corollary \ref{cor1}.

\begin{theorem}\label{ti2a}
Let $d=4, m=1$ and
let $(\Omega,\A,\PP)$ be a probability space. For every $n\in\N$ let
$\psi_{n,1}\colon \Omega\to\R^4$
as well as
\[
\psi_{n,i}\colon (\R^4\times \R^4)^{\N_0\times \{1,\dots,i-1\}} \times \Omega \to \R^4,\quad i=2,\dots,n,
\]
and
\[
\varphi_n \colon  (\R^4\times \R^4)^{\N_0\times \{1,\dots,n\}}\times \Omega \to \R
\]
be measurable mappings. For all infinitely often differentiable functions $a\colon \R^4\to \R^4$ and $b\colon \R^4\to \R^{4}$ and for every $n\in\N$ define
random variables $Z_{n,1}^{a,b},\dots, Z_{n,n}^{a,b}\colon \Omega \to  (\R^4\times \R^4)^{\N_0}$ by $Z_{n,1}^{a,b}(\omega) = \bigl((D^\alpha a,D^\alpha b)(\psi_{n,1}(\omega))\bigr)_{\alpha\in \N_0^4}$ and
\[
Z_{n,i}(\omega) = \bigl((D^\alpha a,D^\alpha b)(\psi_{n,i}(Z_{n,1}(\omega),\dots,Z_{n,i-1}(\omega),\omega))\bigr)_{\alpha\in \N_0^4}, \quad i=2,\dots,n.
\]

Then  for every $n\in\N$ there exist  infinitely often differentiable functions $a_n,b_n\colon \R^4\to \R^4$ with  $\|a_n\|_\infty\le 1$, $ \|b_n\|_\infty\le 1$ and $\|D^\alpha b_n\|_\infty \le 1$,
$\|D^\alpha a_n/(1+ \exp(|\cdot|^3)\|_\infty \le 1$
 for all $\alpha\in \N_0^4$ with $|\alpha|_1=1$
such that for every $\kappa \in (0,\infty)$,
\[
\lim_{n\to\infty} n^\kappa \cdot  \EE\bigl[|\EEE [X^{a_n,b_n}_1(1)] - \varphi_n(Z_{n,1}^{a_n,b_n},\dots, Z_{n,n}^{a_n,b_n}, \cdot) |\bigr] = \infty.
\]
\end{theorem}

  Perhaps even more surprising
we obtain from Corollary~\ref{cor2} in Section~\ref{sub4-1} that for every such sequence of Monte Carlo methods and for every arbitrarily slow convergence speed there exists a strictly increasing and continuous function $u\colon [0,\infty) \to [0,\infty)$
and a sequence of infinitely often differentiable and bounded by one coefficients $a_n$, $b_n$ such that the  order one partial derivatives of $b_n$ are bounded by one as well, the order one partial derivatives of $a_n$ are dominated by the function $ 1 + u(|\cdot|)$ and the resulting sequence of mean absolute errors
for computing the expectation of the first component of the solution at the final time can not converge to zero faster than the given speed of convergence. This finding is formally stated in Theorem~\ref{ti2b}, which follows from Corollary~\ref{cor2} in Section~\ref{sub4-1}.

\begin{theorem}\label{ti2b}
Let $d=4, m=1$ and
let $(\Omega,\A,\PP)$ be a probability space. For every $n\in\N$ let
$\psi_{n,1}\colon \Omega\to\R^4$
and let
\[
\psi_{n,i}\colon (\R^4\times \R^4)^{\N_0\times \{1,\dots,i-1\}} \times \Omega \to \R^4,\quad i=2,\dots,n,
\]
as well as
\[
\varphi_n \colon  (\R^4\times \R^4)^{\N_0\times \{1,\dots,n\}}\times \Omega \to \R
\]
be measurable mappings. For all infinitely often differentiable functions $a\colon \R^4\to \R^4$ and $b\colon \R^4\to \R^{4}$ and for every $n\in\N$ define
random variables $Z_{n,1}^{a,b},\dots, Z_{n,n}^{a,b}\colon \Omega \to  (\R^4\times \R^4)^{\N_0}$ by $Z_{n,1}^{a,b}(\omega) = \bigl((D^\alpha a,D^\alpha b)(\psi_{n,1}(\omega))\bigr)_{\alpha\in \N_0^4}$
 and
\[
Z_{n,i}(\omega) = \bigl((D^\alpha a,D^\alpha b)(\psi_{n,i}
(Z_{n,1}(\omega),\dots,Z_{n,i-1}(\omega),\omega))\bigr)_{\alpha\in \N_0^4},
\quad i=2,\dots,n.
\]
Let $(\eps_n)_{n\in\N}$ be a sequence of positive reals with $\lim_{n\to \infty}\eps_n = 0$.

Then there exist $c\in (0,\infty)$ and a strictly increasing, continuous function $u\colon[0,\infty) \to [0,\infty)$ as well as sequences $(a_n)_{n\in\N}$ and $(b_n)_{n\in\N}$ of  infinitely often differentiable functions $a_n,b_n\colon \R^4\to \R^4$
with  $\|a_n\|_\infty\le 1$, $ \|b_n\|_\infty\le 1$ and
$\| D^\alpha a_n/(1+ u(|\cdot|))\|_\infty \le 1$,
$\|D^\alpha b_n\|_\infty \le 1$
for all $\alpha\in \N_0^4$ with $|\alpha|_1=1$, such that for every $n\in\N$,
\[
 \EE\bigl[|\EEE [X^{a_n,b_n}_1(1)] - \varphi_n(Z_{n,1}^{a_n,b_n},\dots, Z_{n,n}^{a_n,b_n}, \cdot) |\bigr] \ge c\cdot \eps_n.
\]
\end{theorem}

In Theorems~\ref{ti1}--\ref{ti2b} the integrand $f$ is fixed to be a coordinate projection  and lower bounds
  are provided
for the worst case mean absolute error of a  Monte Carlo quadrature
rule on subclasses of equations \eqref{sde} with infinitely often
differentiable coefficients that are bounded by one. On the other
hand, one can fix a specific equation \eqref{sde} with infinitely
often differentiable and bounded coefficients $a$ and $b$  and study
the worst case  mean absolute error of a  Monte Carlo quadrature
rule with respect to a class of integrands $f$. In the latter
setting a negative result of the type stated in Theorems 2 and 3,
which holds for any sequence of Monte Carlo quadrature rules that
are based on finitely many evaluations of the integrand $f$, can of
course not be true. In fact, consider the direct
 simulation method
 $\Sh_n^\text{ds}$ based on $n$ repetitions of the solution $X^{a,b}(1)$ of the fixed equation \eqref{sde} at the final time, i.e.,
\[
\Sh_n^\text{ds}(a,b,f) = \frac{1}{n}\sum_{i=1}^n f(V^{a,b}_i),
\]
where $V^{a,b}_1,\dots,V^{a,b}_n$ are independent and identically distributed as $X^{a,b}(1)$.
Clearly, if $f$ is bounded by one  then
\[
\EE\bigl[|S(a,b,f) - \Sh_n^\text{ds}(a,b,f)|^2\bigr] \le \frac{1}{n}.
\]

However, if only deterministic quadrature rules are considered then we obtain again
negative statements in the spirit of Theorems~\ref{ti2a} and \ref{ti2b} even for the seemingly easy problem of computing the expected value $\EEE[f(W(1))]$
for a one-dimensional Brownian motion $W$ and infinitely often
differentiable integrands $f\colon \R\to \R$ that are bounded by
one. For instance, we can show that for any sequence of
deterministic quadrature rules that are based on evaluations of the
integrand $f$  and all its derivatives at finitely many points in
$\R$ and for every arbitrarily slow convergence speed there exists a
strictly increasing and continuous function $u\colon [0,\infty) \to
[0,\infty)$ and a sequence of infinitely often differentiable and
bounded by one integrands $f_n\colon\R\to\R$ such that the order one
partial derivatives of $f_n$ are dominated by the function $ 1 +
u(|\cdot|)$ and the resulting sequence of approximation errors for
computing the expectation $\EEE [f_n(W(1))]$ can not converge to
zero faster than the given speed of convergence. This finding, which
is formally stated in the following theorem, is a straightforward
consequence of Corollary~\ref{corr4} in Section~\ref{sub4-2}.
\begin{theorem}\label{ti3}
Assume that $W$ is a one-dimensional Brownian motion. For every
$n\in\N$ let $x_{n,1},\dots,x_{n,n}\in \R$ and let $\varphi_n\colon
\R^{\N_0\times \{1,\dots,n\}}\to \R$ be a measurable mapping. Let
$(\eps_n)_{n\in\N}$ be a sequence of positive reals with $\lim_{n\to
\infty}\eps_n = 0$. Then there exists $c\in (0,\infty)$, a strictly
increasing, continuous function $u\colon   [0,\infty)\to [0,
\infty)$ and a sequence of infinitely often differentiable functions
$f_n \colon \R\to \R$ with $\|f_n\|_\infty \le 1$ and
$\|f_n^{(1)}/(1+u(|\cdot|))\|_\infty \le 1$ such that for every
$n\in\N$,
\[
\bigl|\EEE[f_n(W(1))] - \varphi_n\bigl((f_n^{(k)}(x_{n,1}),\dots, f_n^{(k)}(x_{n,n}))_{k\in \N_0} \bigr)\bigr| \ge c\cdot \eps_n.
\]
\end{theorem}

The findings stated in Theorems~\ref{ti1}--~\ref{ti2b} are worst
case results for randomized quadrature rules with respect to a given
class of equations~\eqref{sde}. It remains an open question whether
these results can be strengthened in the sense that for every
sequence of Monte Carlo methods for quadrature of the first
component of the solution, which are based on finitely many
sequential evaluations of the coefficients and all their partial
derivatives, there exists a single equation with infinitely often
differentiable and bounded coefficients, which leads to the
prescribed slow convergence rate of the corresponding sequence of
mean absolute errors. Up to now, a positive answer to this question
is only known for the sequence of Euler Monte Carlo schemes,
see~\cite{hhj12} and \eqref{zz1}. Similarly, it is unclear, whether
Theorem~\ref{ti3} can be strengthened in the sense that for every
sequence of deterministic quadrature rules for quadrature with
respect to the one-dimensional standard normal distribution, which
are based on finitely many sequential evaluations of the integrand
and all its derivatives, there exists a single infinitely often
differentiable and bounded integrand leading to the prescribed slow
convergence rate of the corresponding sequence of absolute errors.
We conjecture that both questions can be answered to the positive
and we will address  these issues in future research.

We add that there is a number of results on worst case lower error bounds for quadrature of marginals of SDEs in the case of coefficients $a, b$ that satisfy a uniform global Lipschitz condition  and integrands $f$ with first order partial derivatives that satisfy a uniform polynomial growth condition, see~\cite{PWW00,Kw03,PR06,MGRY2015}.

We further add that recently in~\cite{JMGY15} equations~\eqref{sde} with infinitely often differentiable and bounded coefficients $a$, $b$ have been constructed that can not be approximated at the final time in the pathwise sense
with a polynomial rate
by any approximation method based on finitely many evaluations of the driving Brownian motion. In the present paper we use a construction, which is conceptually similar to the one from~\cite{JMGY15} but specifically tailored to the analysis of the quadrature problem.

We briefly describe the content of the paper. In Section~\ref{not}
we fix some notation with respect to the regularity of coefficients
and integrands. In Section~\ref{Sec3} we set up the framework for
studying worst case errors of randomized and deterministic
algorithms for the approximation of nonlinear functionals on
function spaces. In particular, we establish lower error bounds for
the corresponding minimal randomized and deterministic errors that
generalize classical results of Bakhvalov~\cite{Bak59} and
Novak~\cite{Nov88} for linear integration problems. In
Section~\ref{Sec4} we use the framework from Section~\ref{Sec3} to
study quadrature problems for SDEs. Section~\ref{sub4-1} is devoted
to  lower bounds for worst case errors with respect to the
coefficients, while Section~\ref{sub4-2} contains our results on
worst case errors with respect to the integrands. The proofs of the
main results, Theorems~\ref{mainThm} and~\ref{mainThm2}, are carried
out in Section~\ref{proof}.

\section{Notation}\label{not}
Let $k,\ell_1,\ell_2\in \N$. For a vector $x\in\R^k$ and a matrix $M\in \R^{\ell_1\times \ell_2}$ we use $|x|$ and $|M|$ to denote the maximum norm of $x$ and $M$, respectively. For a function $h\colon \R^k\to\R^{\ell_1\times \ell_2}$
we put $\|h\|_\infty = \sup_{x\in \R^k}|h(x)|$. By $C^\infty(\R^k;\R^{\ell_1\times \ell_2})$ we denote the set of all functions $h\colon \R^k\to\R^{\ell_1\times \ell_2}$ that are infinitely often differentiable and
for $h\in C^\infty(\R^k;\R^{\ell_1\times \ell_2})$ and a multiindex $\alpha=(\alpha_1,\dots,\alpha_k)\in \N_0^k$  we use
\[
D^\alpha h = \frac{\partial^{\alpha_1+\dots + \alpha_k}h}{\partial x_k^{\alpha_k}\dots \partial x_1^{\alpha_1}}\colon \R^k \to \R^{\ell_1\times \ell_2}
\]
to denote the corresponding partial derivative of $h$. For every $\nu\in\N_0$ we use
\[
C^{\infty,\nu}(\R^k;\R^{\ell_1\times \ell_2}) = \Bigl\{h\in C^\infty(\R^k;\R^{\ell_1\times \ell_2})\colon \max_{\alpha\in \N_0^k, \alpha_1+\dots + \alpha_k \le \nu }\|D^\alpha h\|_\infty \le 1\Bigr\}
\]
to denote all functions $h\in C^\infty(\R^k;\R^{\ell_1\times \ell_2})$ that are  bounded by one and have partial derivatives up to order $\nu$ that are bounded by one as well.

\section{Approximation of nonlinear functionals on function spaces and lower worst case error bounds}\label{Sec3}
Let $A$ and $B$ be nonempty sets, let $\ccF\subset B^A$ be a nonempty set of functions $g\colon A \to B$ and let
\[
S\colon \ccF \to \R.
\]
We study the approximation of $S(\ff)$ for $\ff\in\ccF$ by means of
a  deterministic or randomized algorithm that is based on finitely
many evaluations of the mapping $\ff$ at points in $A$. Our goal is
to provide lower bounds for the worst case mean error of any such
algorithm in terms of its worst case average number of function
evaluations.

A generalized randomized algorithm for this problem is specified by
a probability space $(\Omega,\mathcal{A},\PP)$ and a triple
\[
(\psi,\nu,\varphi),
\]
where
\begin{itemize}
\item $\psi=(\psi_k)_{k\ge 1}$ is a sequence of mappings
\[
\psi_k\colon B^{k-1}\times \Omega \to A,
\]
which are used to sequentially determine random evaluation nodes in $A$ for a given input $\ff\in\ccF$,
\item the mapping
\[
\nu\colon \ccF\times \Omega \to \N
\]
determines the random total number of evaluations of a given input $\ff\in \ccF$, and
\item $\varphi=(\varphi_k)_{k\ge 1}$ is a sequence of mappings
\[
\varphi_k\colon B^{k}\times \Omega \to \R,
\]
which are used to obtain for every input $\ff\in \ccF$ a random approximation to $S(\ff)$ based on the observed function values of $\ff$.
\end{itemize}
To be more precise, we
define for every $k\in\N$ a mapping
\[
N_k^\psi\colon \ccF \times \Omega \to B^{k}
\]
  by
\[
N_k^\psi = (y_1,\dots,y_k),
\]
where
\[
y_1(\ff,\omega) = \ff(\psi_1(\omega))
\]
and
\[\phantom{\ell=2,\dots,k.}
y_\ell(\ff,\omega) = \ff\bigl(\psi_\ell(y_1(\ff,\omega),\dots,y_{\ell-1}(\ff,\omega),\omega)\bigr), \quad \ell=2,\dots,k.
\]
For a  given $\omega\in\Omega$ and a given input $\ff\in\ccF$
the algorithm specified by $(\psi,\nu,\varphi)$ sequentially performs $\nu(\ff,\omega)$ evaluations of  $\ff$ at the points
\[
\psi_1(\omega),\, \psi_2(y_1(\ff,\omega), \omega),
\dots, \psi_{\nu(\ff,\omega)}(y_1(\ff,\omega),\dots,y_{\nu(\ff,\omega)-1}(\ff,\omega), \omega)\in A
\]
and finally applies the mapping $\varphi_{\nu(\ff,\omega)}(\cdot,\omega)\colon B^{\nu(\ff,\omega)}\to \R$ to the data $N^\psi_{\nu(\ff,\omega)}(\ff,\omega)$ to
obtain the real number
\[
\Sh_{\psi,\nu,\varphi}(\ff,\omega) = \varphi_{\nu(\ff,\omega)}\bigl(N^\psi_{\nu(\ff,\omega)}(\ff,\omega),\omega\bigr)
\]
as an approximation to $S(\ff)$. The induced mapping
\[
\Sh_{\psi,\nu,\varphi}\colon \ccF \times \Omega \to \R
\]
is called a generalized randomized algorithm if for every $\ff\in\ccF$ the mappings
\[
\Sh_{\psi,\nu,\varphi}(\ff,\cdot)\colon\Omega\to \R\mbox{ \ and \ }
 \nu(\ff,\cdot)\colon\Omega\to \N
\]
are random variables.

We use  $\Aran$ to denote the
class of all randomized algorithms. The error and the cost of $\Sh\in\Aran$ are defined in the worst case sense by
\[
e(\Sh) = \sup_{\ff\in \ccF} \E|S(\ff)- \Sh(\ff,\cdot)|
\]
and
\[
\cost(\Sh) = \inf_{\psi,\nu,\varphi}\biggl\{\sup_{\ff\in\ccF} \E\,\nu(\ff,\cdot)\colon\,\Sh = \Sh_{\psi,\nu,\varphi}\biggr\},
\]
respectively. Thus the definition of the  cost of $\Sh$ takes into account that  the representation $\Sh=\Sh_{\psi,\nu,\varphi}$ is not unique in general.

A generalized randomized algorithm $\Sh\in \Aran$
is called deterministic if the random variable $\Sh(\ff,\cdot)$ is constant for all $\ff\in \ccF$. In this case we have $\Sh = \Sh_{\psi,\nu,\varphi}$ with mappings
\begin{equation}\label{deter2}
\psi_k\colon B^{k-1}\to A, \,\, \nu\colon  \ccF\to \N,\,\, \varphi_k\colon B^{ k}\to \R,
\end{equation}
and it is easy to see that
\[
\cost(\Sh) = \inf\biggl\{\sup_{\ff\in\ccF} \nu(\ff)\colon\,\Sh = \Sh_{\psi,\nu,\varphi},(\psi,\nu,\varphi) \text{ satisfies \eqref{deter2}} \biggr\}.
\]
The class of all generalized deterministic algorithms is denoted by
$\Adet$.

Let $n\in\N$. The crucial quantities for our analysis are the $n$-th minimal errors
\[
\edet_n(\ccF; S)  = \inf\{e(\Sh)\colon \Sh\in \Adet,\, \cost(\Sh)\le n\}
\]
and
\[
\eran_n(\ccF; S)  = \inf\{e(\Sh)\colon \Sh\in \Aran,\, \cost(\Sh)\le n\},
\]
i.e., the smallest possible worst case error that can be achieved by generalized
deterministic algorithms based on at most $n$ function values of $\ff\in \ccF$ and the smallest possible worst case mean error that can be achieved by generalized  randomized  algorithms that use at most $n$ function values of $\ff\in \ccF$ on average, respectively. Clearly, $\edet_n(\ccF; S)\ge \eran_n(\ccF; S)$.

We present two types of lower bounds for the minimal errors $\edet_n(\ccF; S)$ and $\edet_n(\ccF; S)$, which generalize classical results of Bakhvalov and Novak for
the case of $S$ being a linear functional on a space  $\ccF$ of real-valued functions $g\colon A\to \R$,    see ~\cite{Bak59,Nov88}.
\begin{prop}\label{prop1}
Let $\varepsilon >0$, $m\in\N$, $b^*\in B$ and assume that there exist $2m$ functions
\[
\ff_{1,+},\ff_{1,-},\dots, \ff_{m,+},\ff_{m,-}\colon A\to B
\]
with the following properties.\\[-.3cm]
\begin{itemize}
\item[(i)] The sets
\[
\{\ff_{1,+}\neq b^*\}\cup\{\ff_{1,-}\neq b^*\},\dots,\{\ff_{m,+}\neq b^*\}\cup\{\ff_{m,-}\neq b^*\}
\]
 are pairwise disjoint,\\[-.3cm]
\item[(ii)] We have $\ff_{1,+},\ff_{1,-},\dots, \ff_{m,+},\ff_{m,-}\in \ccF$,\\[-.3cm]
\item[(iii)] We have $S(\ff_{i,+})-S(\ff_{i,-})\ge \varepsilon$ for $i=1,\dots,m$.\\[-.3cm]
\end{itemize}
Then, for every $n\in\N$,
\[
e_n^{\text{ran}}(\ccF; S) \ge \frac{m-16n}{8m}\,\varepsilon.
\]
\end{prop}

\begin{prop}\label{prop2}
Let $B$ be a linear space. Let $\varepsilon >0$, $m\in \N$, $b^*\in B$ and assume that there exist $2m$ functions
\[
\ff_{1,+},\ff_{1,-},\dots, \ff_{m,+},\ff_{m,-}\colon A\to B
\]
 with the following properties.\\[-.3cm]
\begin{itemize}
\item[(i)] The sets
\[
\{\ff_{1,+}\neq b^*\}\cup\{\ff_{1,-}\neq b^*\},\dots,\{\ff_{m,+}\neq b^*\}\cup\{\ff_{m,-}\neq b^*\}
\]
 are pairwise disjoint,\\[-.3cm]
\item[(ii)] We have $\ff_{1,+},\ff_{1,-},\dots, \ff_{m,+},\ff_{m,-}\in \ccF$ and for all $\delta_1,\dots,\delta_{m}\in\{+,-\}$ we have
\[
\sum_{i=1}^m \ff_{i,\delta_i}\in \ccF \quad \text{and} \quad S\Bigl(\sum_{i=1}^m \ff_{i,\delta_i}\Bigr) = \sum_{i=1}^m S(\ff_{i,\delta_i}),
\]
\item[(iii)] We have  $S(\ff_{i,+})-S(\ff_{i,-})\ge \varepsilon$ for $i=1,\dots,m$.\\[-.3cm]
\end{itemize}
Then, for every $n\in\N$,
\[
e_n^{\text{det}}(\ccF; S) \ge \frac{m-n}{2}\,\varepsilon
\]
and for every $n\leq m/4$,
\[
e_n^{\text{ran}}(\ccF; S) \ge \sqrt{\frac{m-4n}{128}} \,\varepsilon.
\]
\end{prop}

For the proof of the lower bounds for the $n$-th minimal randomized
errors in Propositions \ref{prop1} and \ref{prop2}  we employ a
classical averaging principle of Bakhvalov, see \cite{Bak59}.
Consider a probability measure $\mu$ on the power set $\cP(\ccF)$ of
$\ccF$ with finite support. For a deterministic algorithm $\Sh\in
\Adet$ we define the average error and the average cost of $\Sh$
with respect to $\mu$ by
\[
e(\Sh,\mu) = \int_\ccF  |S(\ff)-\Sh(\ff)|\,\mu(d\ff)
\]
and
\[
\cost(\Sh,\mu) = \inf \biggl\{\int_\ccF \nu(\ff)\,\mu(d\ff) \colon \Sh=\Sh_{\psi,\nu,\varphi}, (\psi,\nu,\varphi) \text{ satisfies \eqref{deter2}}\biggr\}.
\]

The smallest possible average case error with respect to $\mu$ that can be achieved by any generalized deterministic algorithm based on at most $n$ function evaluations on average with respect to $\mu$ is then given by
\[
e_n^\text{det}(\mu) = \inf\{e(\Sh,\mu)\colon \Sh\in \Adet,\, \cost (\Sh,\mu) \le n\}.
\]

\begin{lemma}\label{lemA3}
For every  probability measure $\mu$ on  $\cP(\ccF)$ and every $n\in\N$ we have
\[
e_n^\text{ran}(\ccF; S)\ge \frac{1}{2} \, e_{2n}^\text{det}(\mu).
\]
\end{lemma}

For convenience of the reader we provide a proof of Lemma \ref{lemA3}

\begin{proof}[Proof of Lemma \ref{lemA3}]
Let $\Sh\in \Aran$ with $\cost(\Sh)\le n$. Let $\rho>0$ and choose
$(\psi,\nu,\varphi)$ such that $\Sh=\Sh_{\psi,\nu,\varphi}$ and
$\sup_{g\in \ccF}\E [\nu(g,\cdot)] \le n+\rho$. Put
\[
\Omega_1 = \Bigl\{\omega\in \Omega\colon \int_{\ccF} \nu(g,\omega)\, d\mu(g) \le 2n\Bigr\}.
\]
Then $2n\, \PP(\Omega_1^c) \le \sup_{g\in \ccF}\E [\nu(g,\cdot)]\le n+\rho$, and therefore, $\PP(\Omega_1)\ge 1/2-\rho/(2n)$. For every $\omega\in \Omega_1$ we have $\Sh(\cdot, \omega)\in \Adet$ and $\cost(\Sh(\cdot, \omega),\mu)\le 2n$, which implies
\[
\int_\ccF |S(g)-\Sh(g,\omega)|\, \mu(dg) \ge e_{2n}^\text{det}(\mu).
\]
Hence
\begin{align*}
e(\Sh) & \ge \int_\ccF \E[|S(g)-\Sh(g,\cdot)|]\, \mu(dg)\\ & \ge \int_{\Omega_1}\int_\ccF |S(g)-\Sh(g,\omega)|\,\mu(dg)\,\PP(d\omega)\ge (1/2-\rho/(2n))\cdot e_{2n}^\text{det}(\mu)
\end{align*}
Letting $\rho$ tend to zero completes the proof.
\end{proof}

\begin{proof}[Proof of Proposition \ref{prop1}]
Let $\mu$ denote the uniform distribution on
\[
\ccFt = \{\ff_{1,+},\ff_{1,-},\dots, \ff_{m,+},\ff_{m,-}\}.
\]
 We show that
\begin{equation}\label{ee1}
e_{n}^\text{det}(\mu) \ge \frac{m-8n}{4m}\,\varepsilon,
\end{equation}
which jointly with Lemma \ref{lemA3}   yields the lower bound in Proposition \ref{prop1}.

In order to prove \eqref{ee1}, let $\Sh\in \Adet$ with $\cost (\Sh,\mu) \le n$. Let $\rho >0$ and choose $(\psi,\nu,\varphi)$ satisfying \eqref{deter2} such that $\Sh=\Sh_{\psi,\nu,\varphi}$ and $\int_\ccF \nu(\ff)\,
\mu(d\ff) \le n+\rho$. Put
\[
\ccFt_1 = \{\ff\in \ccFt\colon \nu(\ff) \le 4n\}
\]
and let
\[
I = \{i\in \{1,\dots,m\}\colon \ff_{i,+},\ff_{i,-}\in \ccFt_1 \}.
\]
Then $4n\, \mu(\ccFt_1^c) \le 
n +\rho $,
and therefore $\mu(\ccFt_1) \ge 3/4 -\rho/(4n)$. Since $|\ccFt_1| = \mu(\ccFt_1)\cdot 2m $, we conclude that
\[
|I| \ge (1/2 -\rho/(2n))\cdot m.
\]
Let
\[
K = \{\psi_1,\psi_2(b^*),\dots, \psi_{4n}(b^*,\dots,b^*)\}
\]
denote the set of the first $4n$  nodes in $A$ that are produced by the sequence $(\psi_k)_{k\in \N}$ for evaluating the constant function $x\mapsto b^*$ on $A$, and put
\[
J = \{i\in I\colon K\cap (\{\ff_{i,+}\neq b^*\}\cup\{\ff_{i,-}\neq b^*\}) = \emptyset\}.
\]
Clearly, $\Sh_{\psi,\nu,\varphi}(\ff_{i_1,\delta_1}) = \Sh_{\psi,\nu,\varphi}(\ff_{i_2,\delta_2})$ for all $i_1,i_2\in J$ and all $\delta_1,\delta_2\in\{-,+\}$, and, observing property (i), we conclude $|J| \ge (1/2 -\rho/(2n))m - 4n$. Thus, by property (iii),
\begin{align*}
\int_\ccF |S(\ff)-\Sh(\ff)|\,\mu(d\ff) & \ge \frac{1}{2m} \sum_{i\in J} (|S(\ff_{i,+})-\Sh(\ff_{i,+})| + |S(\ff_{i,-})-\Sh(\ff_{i,-})|)\\
& \ge \frac{1}{2m}\sum_{i\in J} |S(\ff_{i,+})-S(\ff_{i,-})|\\
& \ge  \frac{|J|}{2m} \,\varepsilon\ge  (1/4 - \rho/(4n) - 2n/m)\,\varepsilon.
\end{align*}
Letting $\rho$ tend to zero yields
\[
e(\Sh,\mu) \ge (1/4 -2n/m)\,\varepsilon,
\]
which  completes the proof.
\end{proof}

\begin{proof}[Proof of Proposition \ref{prop2}]
We first prove the lower bound for the $n$-th minimal error of deterministic methods. Let $\Sh\in \Adet$ with $\cost (\Sh) \le n$. Choose $(\psi,\nu,\varphi)$ satisfying \eqref{deter2} such that $\Sh=\Sh_{\psi,\nu,\varphi}$ and $\sup_{g\in \ccF} \nu(g)\, \le n$. Consider the function
\[
g = \sum_{i=1}^m g_{i,+} \in \ccF
\]
and let $K$ denote the set of at most $n$ nodes in $G$ that are used by $\Sh_{\psi,\nu,\varphi}$ for evaluating $g$. Put
\[
J = \{i\in \{1,\dots,m\}\colon K\cap (\{g_{i,+}\neq b^*\}\cup\{g_{i,-}\neq b^*\}) = \emptyset\}
\]
and let
\[
h = \sum_{i\not\in J}g_{i,+} + \sum_{i\in J}g_{i,-}.
\]
Clearly, $\Sh_{\psi,\nu,\varphi}(g) = \Sh_{\psi,\nu,\varphi}(h)$, and, observing property (i), $|J| \ge m-n$. Thus, by properties (ii) and (iii),
\begin{align*}
e(\Sh) & \ge \frac{1}{2}\bigl(|S(g)-\Sh(g)|+|S(h)-\Sh(h)|\bigr)\ge \frac{1}{2}|S(g)-S(h)|\\
& \ge \frac{1}{2}\sum_{i\in J}(S(g_{i,+})- S(g_{i,-})) \ge \frac{|J|}{2}\varepsilon \ge (m-n)\cdot \varepsilon/2,
\end{align*}
which completes the proof of the lower bound for the minimal deterministic error.

We turn to the proof of the lower bound for the minimal randomized error.
Let $\mu$ denote the uniform distribution on
\[
\ccFt = \Bigl\{\sum_{i=1}^m g_{i,\delta_i}\colon \delta_1,\dots,\delta_m\in \{+,-\}\Bigr\}.
\]
 We show that
\begin{equation}\label{ee1b}
e_{n}^\text{det}(\mu) \ge \sqrt{\frac{m-2n}{32}}\,\varepsilon,
\end{equation}
which jointly with Lemma \ref{lemA3}   yields the
   desired
lower bound in Proposition \ref{prop2}.

In order to prove \eqref{ee1b}, let $\Sh\in \Adet$ with $\cost (\Sh,\mu) \le n$. Let $\rho >0$ and choose $(\psi,\nu,\varphi)$ satisfying \eqref{deter2} such that $\Sh=\Sh_{\psi,\nu,\varphi}$ and $\int_\ccF \nu(f)\,\mu(df) \le n+\rho$. Put
\[
\ccFt_1 = \{g\in \ccFt\colon \nu(g) \le 2n\}.
\]
Then $2n\, \mu(\ccFt_1^c) 
\le n +\rho $
and therefore $\mu(\ccFt_1) \ge 1/2 -\rho/(2n)$. Hence
\begin{equation}\label{ee4a}
|\ccFt_1| = \mu(\ccFt_1)\cdot 2^m \ge (1- \rho/n)\cdot 2^{m-1}.
\end{equation}

Consider the function $N^\psi_{2n}\colon \ccF\to B^{2n}$ and put
\[
\cY = \{N^\psi_{2n}(g)\colon g\in \ccFt_1\}
\]
as well as
\[
\ccFt_{1,y} = \{g\in \ccFt_1\colon N^\psi_{2n}(g) =y\}
\]
for all $y\in \cY$. Note that
\[
\ccFt_1 = \bigcup_{y\in\cY} \ccFt_{1,y}.
\]

Fix $y\in\cY$, let $K_y$ denote the set of the nodes in $A$ that are used by $\Sh_{\psi,\nu,\varphi}$ for evaluating each of the functions $g\in \ccFt_{1,y}$ and put
\[
J_y = \{i\in\{1,\dots,m\}\colon K_y\cap (\{g_{i,+}\neq b^*\}\cup \{g_{i,-}\neq b^*\})=\emptyset\}.
\]
For $g=\sum_{i=1}^m g_{i,\delta_i}\in \ccFt_{1,y}$ let $g_1 = \sum_{i\not\in J_y}g_{i,\delta_i}$ and put
\[
\ccFt_{1,y}(g_1) = \Bigl\{g_1+\sum_{i\in J_y}g_{i,\tilde \delta_i}\colon (\tilde \delta_i)_{i\in J_y}\in\{-,+\}^{|J_y|}\Bigr\}.
\]
Let $\ccFt^y = \{g_1\colon g\in \ccFt_{1,y}\}$. Clearly,
$\ccFt_{1,y}(g_1) \cap \ccFt_{1,y}(h_1) =\emptyset$ for all
$g_1,h_1\in \ccFt^y$ with $g_1\neq h_1$ and
\[
\ccFt_{1,y} = \bigcup_{g_1\in \ccFt^y}\ccFt_{1,y}(g_1).
\]
We show that for every $h\in \ccFt_{1,y}$,
\begin{equation}\label{ee3a}
\sum_{g\in \ccFt_{1,y}(h_1)}|S(g)-\Sh(g)| \ge 2^{|J_y|-3/2}\,(m-2n )^{1/2}\, \varepsilon.
\end{equation}
Using \eqref{ee4a} and \eqref{ee3a} we may then conclude that
\begin{align*}
\int_\ccF |S(g)-\Sh(g)|\,\mu(dg) &\ge 2^{-m}\sum_{g\in \ccFt_1} |S(g)-\Sh(g)|= 2^{-m} \sum_{y\in \cY}\sum_{h_1\in \ccFt^y} \sum_{g\in \ccFt_{1,y}(h_1)} |S(g)-\Sh(g)|  \\
&  \ge 2^{-m} \sum_{y\in \cY}\sum_{h_1\in \ccFt^y} 2^{|J_y|-3/2}\,(m-2n )^{1/2}\, \varepsilon \\
& = 2^{-m-3/2} |\ccFt_1|\,(m-2n )^{1/2}\, \varepsilon \\
& \ge 2^{-5/2}(1-\rho/n )\,(m-2 n)^{1/2} \, \varepsilon.
\end{align*}
Letting $\rho$ tend to zero yields
\[
e(\Sh,\mu) \ge 2^{-5/2}\,(m-2n)^{1/2}\,\varepsilon,
\]
which in turn implies \eqref{ee1b}.

It remains to prove the estimate  \eqref{ee3a}.
For $\delta\in\{+,-\}$ we define
\[
\bar \delta = \begin{cases}
+, & \text{if }\delta = -,\\ -,& \text{if }\delta = +,
\end{cases}
\]
and for $g=h_1+\sum_{i\in J_y}g_{i,\delta_i}\in \ccFt_{1,y}(h_1)$ we put
\[
\bar g = h_1+\sum_{i\in J_y}g_{i,\bar \delta_i}.
\]
Clearly,  $\Sh_{\psi,\nu,\varphi}(g) = \Sh_{\psi,\nu,\varphi}(\bar
g)$ for all $g\in \ccFt_{1,y}(h_1)$. Thus, by property (ii), the
Khintchine inequality, see~\cite{Sz1976},
 and property (iii),
\begin{align*}
\sum_{g\in \ccFt_{1,y}(h_1)}|S(g)-\Sh(g)| & = \frac{1}{2} \sum_{g\in \ccFt_{1,y}(h_1)} (|S(g)-\Sh(g)|+|S(\bar g)-\Sh(\bar g)|) \\
& \ge \frac{1}{2} \sum_{g\in \ccFt_{1,y}(h_1)} |S(g)-S(\bar g)|\\
& = \frac{1}{2} \sum_{\delta\in \{-,+\}^{|J_y|}}\Bigl|\sum_{i\in J_y} (S(g_{i,\delta_i})- S(g_{i,\bar \delta_i}))\Bigr|\\
& = \frac{1}{2} \sum_{\delta\in \{-1,+1\}^{|J_y|}}\Bigl|\sum_{i\in J_y} \delta_i (S(g_{i,+})- S(g_{i,-}))\Bigr|\\
& \ge \frac{2^{|J_y|-1/2}}{2}  \Bigl(\sum_{i\in J_y}  (S(g_{i,+})- S(g_{i,-}))^2\Bigr)^{1/2}\\
& \ge 2^{|J_y|-3/2}\,|J_y|^{1/2}\,\varepsilon.
\end{align*}
Using $|J_y|\ge m-2n$ completes the proof of \eqref{ee3a} and hereby finishes the proof of Proposition~\ref{prop2}.
\end{proof}

\section{Lower worst case error bounds for quadrature  of SDEs}\label{Sec4}

We consider a class of equations \eqref{sde} specified by a class
\[
\eq\subset C^{\infty,0}(\R^d;\R^d)\times C^{\infty,1}(\R^d;\R^{d\times m})
\]
of coefficients $(a,b)$ and a class
\[
\cF \subset C^\infty(\R^d;\R)
 \]
of integrands $f$ satisfying a polynomial growth condition, and we study the problem of approximating
\[
\EEE [f(X^{a,b}(1))]
\]
for all $(a,b)\in\eq$ and all $f\in\cF$ by means of a randomized or
   a deterministic algorithm that may use function values of $a$, $b$, $f$ and all
partial derivatives $D^\alpha a$, $D^\alpha b$, $D^\alpha f$ at finitely many points in $\R^d$. Our goal ist to establish a lower bound for the smallest possible worst case mean error over the classes $\eq$ and $\cF$ that can be achieved by any such algorithm if, on average, at most $n$ evaluation nodes in $\R^d$ may be used.

We formalize this problem in terms of the framework specified in Section \ref{Sec3} as follows.
Put
\[
A=\R^d, \quad B=(\R^d\times \R^{d\times m}\times \R)^{\N_0^d},
\]
let $\ccF(\eq,\cF) \subset B^A$ be given by
\[
\ccF(\eq,\cF)= \{(D^\alpha a,D^\alpha b,D^\alpha f)_{\alpha\in \N_0^d}\colon (a,b)\in \eq,\, f\in \cF\}
\]
and define $\sSDE\colon \ccF(\eq,\cF) \to\R$ by
\[
\sSDE\bigl((D^\alpha a,D^\alpha b,D^\alpha f)_{\alpha\in \N_0^d}\bigr)=\EEE [f(X^{a,b}(1))].
\]
The corresponding classes of deterministic and randomized algorithms  for approximating $\EEE [f(X^{a,b}(1))]$ based on finitely many sequential evaluations of $(D^\alpha a,D^\alpha b,D^\alpha f)_{\alpha\in \N_0^d}$ are given by $\Adet$ and $\Aran$, respectively, and the resulting minimal errors are denoted by $e_n^{\text{det}}(\ccF(\eq\times\cF); \sSDE )$ for deterministic methods and $e_n^{\text{ran}}(\ccF(\eq\times\cF); \sSDE )$ for randomized methods. Note that $\Aran$ contains in particular any Monte Carlo method or multilevel Monte Carlo method that is based on a strong or weak It\^{o}-Taylor scheme of
    arbitrary
order.

\subsection{Worst case analysis with respect to coefficients}\label{sub4-1}
In this section we take $d=4$, $m=1$ and
we consider the following two types of classes of equations $\eq$.
For a function $u\colon [0,\infty) \to [0,\infty)$ we put
\[
\eq_u = \Bigl\{(a,b)\in C^{\infty,0}(\R^4;\R^4)\times
C^{\infty,1}(\R^4;\R^{4})\colon \max_{\alpha\in \N_0^4,
|\alpha|_1 = 1} \sup_{x\in\R^4}\frac{|D^\alpha a(x)|}{1 + u(|x|)}\le
1\Bigr\}.
\]
Thus, $u$ is used to impose a  growth condition on the first order partial derivatives of a drift coefficient $a$. Furthermore,
we consider the class of equations
\[
\eq_{\text{lin}} = \Bigl\{(a,b)\in C^{\infty,0}(\R^4;\R^4)\times C^{\infty,1}(\R^4;\R^{4})\colon\,  \max_{\alpha\in \N_0^4, |\alpha|_1 = 1} \sup_{x\in \R^4} \frac{|D^\alpha a(x)|}{1 + |x|} < \infty\Bigr\},
\]
where the drift coefficient $a$ is required to satisfy a linear growth condition.
Clearly,
\[
\eq_{\text{lin}} = \bigcup_{k=1}^\infty \eq_{u_k},
\]
where $u_k(x) = k\cdot (1+ x)$ for $x\in [0,\infty)$.

The class of integrands is given by
\[
\cF=\{\pi_1\},
\]
where
\[
\pi_1\colon \R^4\to\R,\quad (x_1,\dots,x_4)\mapsto x_1,
\]
is the projection on the first coordinate. We thus study the computation of $\EEE[X^{a,b}_1(1)]$ for all $(a,b)\in \eq_u$  or all $(a,b)\in \eq_\text{lin}$.

The following result provides lower bounds for the minimal errors
$e_n^{\text{ran}}(\ccF(\eq_u\times\{\pi_1\}); \sSDE )$ in case that
the function $u$ is continuous, strictly increasing and satisfies
the condition $\displaystyle{\liminf_{x\to\infty} u(x)/x>0}$. See
Section~\ref{proof1} for a proof.

\begin{theorem}\label{mainThm}
There exist $c_1,c_2\in (0,\infty)$ and $c_3\in [1,\infty)$ such that for all continuous and strictly
increasing functions  $u\colon [0,\infty) \to [0,\infty)$, for all
$\delta \in (0,\infty)$, for all $\xl\in (0,\infty)$ with $\inf_{x\ge \xl} u(x)/x \ge \delta$ and for all $n\in\N$,
\[
e_n^{\text{ran}}\bigl(\ccF(\eq_u\times\{\pi_1\}); \sSDE \bigr)\geq c_1\cdot
 \exp\bigl(- c_2 \cdot \bigl(u^{-1}(c_3 \cdot \delta^{-1} \cdot \max(1,u^2(x_\delta+1)) \cdot n^4)\bigr)^2\bigr).
\]
\end{theorem}

Due to Lemma 5 in \cite{JMGY15} we have
\[
\forall q, c\in (0,\infty)\colon\,
\lim_{n\to\infty}      n^q\cdot \exp\bigl(-c\cdot(u^{-1}(c\cdot
n^4))^2\bigr)=\infty
\]
if and only if
\begin{equation}\label{101}
\forall q\in (0,\infty)\colon\,\lim_{x\to\infty} u(x)\cdot \exp(-q\cdot x^2)=\infty.
\end{equation}
Clearly, \eqref{101} implies $\displaystyle{\liminf_{x\to\infty} u(x)/x>0}$.
As an immediate consequence of Theorem \ref{mainThm} we thus get a non-polynomial decay of the minimal errors $e_n^{\text{ran}}(\ccF(\eq_u\times \{\pi_1\}); \sSDE)$ if $u$ satisfies the exponential growth condition~\eqref{101}.
\begin{cor}\label{cor1}
Assume that $u\colon [0,\infty) \to [0,\infty)$ is
continuous, strictly increasing and satisfies  \eqref{101}. Then for
all $q>0$,
\[
\lim_{n\to\infty} n^q\cdot e_n^{\text{ran}}\bigl(\ccF(\eq_u\times\{\pi_1\}); \sSDE\bigr)=\infty.
\]
\end{cor}

The following result shows that the minimal errors
$e_n^{\text{ran}}(\ccF(\eq_u\times\{\pi_1\}); \sSDE)$ may decay
arbitrary slow.

\begin{cor}\label{cor2}
For every sequence  $(\eps_n)_{n\in\N}\subset (0, \infty)$ with
$\lim_{n\to\infty} \eps_n=0$ there exists  $c\in (0,\infty)$ and a
strictly increasing and continous function $u\colon [0,\infty) \to [0, \infty)$
such that for all $n\in\N$ we have
\[
e_n^{\text{ran}}\bigl(\ccF(\eq_u\times\{\pi_1\}); \sSDE\bigr)\geq c\cdot \eps_n.
\]
\end{cor}

\begin{proof}
Without loss of generality we may assume that the sequence $( \eps_n )_{ n \in \N }$
is strictly decreasing.

Choose $c_1,c_2\in (0,\infty)$ and $c_3\in [1,\infty)$ according to Theorem~\ref{mainThm}.
Choose  $n_0 \in \N$
such that
for all
$ n \geq n_0$
we have
\begin{equation}\label{asss}
  \eps_n \le \exp (-4 c_2)
\end{equation}
and put
\[
  b_n = \Bigl(\frac{ 1 }{ c_2 }\cdot \ln \frac{1}{\eps_n }
  \Bigr)^{1/2}
\]
for $n\geq n_0$.
Note that
$
  ( b_n )_{ n\geq n_0 }
$
is strictly increasing and satisfies
$
  \lim_{  n \to \infty } b_n = \infty
$.

Next, define $u \colon [0, \infty) \to [0,\infty)$ recursively by
\begin{align*}
  u(x) =
  \begin{cases}
    x
    ,
  &
    \text{if }
    x \leq  b_{ n_0}
    ,
\\ u(b_n)+(x-b_n)\cdot \max\bigl(1,\tfrac{c_3\cdot 4(n+1)^4-u(b_n)}{b_{n+1}-b_n}\bigr)
  ,
  &
  \text{if }
  x \in ( b_n , b_{n+1} ] \text{ and } n \geq  n_0
  .
\end{cases}
\end{align*}
Then $ u $ is continuous, strictly increasing and satisfies
\[
  \inf_{ x >0} u( x )/x\geq 1.
\]
Moreover, for all $n\geq n_0+1$ we have
\begin{equation}\label{assu2}
  u(b_n)\geq 4c_3\cdot n^4.
\end{equation}

Note that $b_{n_0}\geq 2$ due to~\eqref{asss}. Hence $u(2)=2$.
Applying Theorem~\ref{mainThm} with $\delta=1$ and $x_\delta=1$ and observing~\eqref{assu2} as well as the fact that $u^{-1}$ is increasing we thus obtain for $n\ge n_0+1$ that
\begin{align*}
e_n^{\text{ran}}\bigl(\ccF(\eq_u\times\{\pi_1\}); \sSDE \bigr) & \ge
c_1\cdot \exp\bigl(-c_2\cdot (u^{-1}(c_3\cdot 4 \cdot  n^4))^2\bigr)\\
& \ge c_1\cdot \exp (-c_2\cdot b_n^2) = c_1 \cdot \eps_n.
\end{align*}
Finally, for all $ n \in \{ 1, 2, \dots, n_0 \}$,
\begin{align*}
e_{n}^{\text{ran}}\bigl(\ccF(\eq_u\times\{\pi_1\}); \sSDE \bigr) & \ge
e_{n_0+1}^{\text{ran}}\bigl(\ccF(\eq_u\times\{\pi_1\}); \sSDE \bigr)\\
&  \ge  c_1 \cdot \eps_{ n_0+1 } > \frac{c_1 \,\eps_{n_0+1}}{\eps_1}\cdot \eps_n,
\end{align*}
which completes the proof.
\end{proof}

As a further consequence of Theorem~\ref{mainThm} it turns out that
the class of equations $\eq_\text{lin}$ is too large to obtain
convergence of the corresponding minimal errors to zero at all.

\begin{cor}\label{cor3}
There exists $c\in (0,\infty)$ such that for all $n\in\N$ we have
\[
e_n^{\text{ran}}\bigl(\ccF(\eq_\text{lin}\times\{\pi_1\}); \sSDE\bigr)\geq c.
\]
\end{cor}

\begin{proof}
Choose $c_1,c_2\in (0,\infty)$ and $c_3\in [1,\infty)$  according to Theorem~\ref{mainThm}.
Let  $n\in\N$. Define $u_n\colon [0,\infty)\to [0,\infty)$ by
\[
u_n(x) = \begin{cases}
c_3\cdot x, & \text{if }x\in [0,2],\\
2c_3  + (4c_3^2\cdot n^4 - 2c_3)\cdot (x-2), & \text{if }x\in (2,\infty).
\end{cases}
\]
Clearly, $u_n$ is continuous and strictly increasing, and for all $x\in [0,\infty)$ we have
\[
u_n(x) \le 2c_3 + 4c_3^2\cdot n^4 \cdot x.
\]
Hence $\eq_{u_n}\subset \eq_\text{lin}$, and therefore
\[
e_n^{\text{ran}}\bigl(\ccF(\eq_\text{lin}\times\{\pi_1\});
\sSDE\bigr) \geq
e_n^{\text{ran}}\bigl(\ccF(\eq_{u_n}\times\{\pi_1\}); \sSDE \bigr).
\]
Clearly, $u_n(x)\ge c_3\cdot x$ for all $x\in [0,\infty)$. Moreover, we have $u_n(2)=2c_3$ and
$u_n(3)=4c_3^2\cdot n^4$.
Applying Theorem \ref{mainThm} with $\delta = c_3$ and $x_{\delta}=1$ we obtain
\begin{align*}
e_n^{\text{ran}}\bigl(\ccF(\eq_{u_n}\times\{\pi_1\}); \sSDE \bigr)
&  \ge c_1\cdot \exp\bigl(-c_2\cdot (u_n^{-1}( \max(1,u_n^2(2))\cdot n^4))^2\bigr) \\
& = c_1\cdot \exp\bigl(-c_2\cdot (u_n^{-1}(4c_3^2\cdot n^4))^2\bigr)= c_1\cdot \exp (-9c_2),
\end{align*}
which completes the proof.
\end{proof}

\begin{rem}
Negative results of the type stated in
Corollaries \ref{cor1}--\ref{cor3} can not hold in the case of ordinary differential equations, i.e., the presence of a
stochastic part is essential to have the described slow convergence phenomena.

In fact, let $d\in\N$,
let $a\in C^{\infty,0}(\R^d;\R^d)$ and consider the ordinary differential equation
\begin{align*}
dX^{a}(t) & = a(X^{a}(t)) \, dt, \quad t\in [0,1],\\
X^{a}(0) & = 0.
\end{align*}
Let $n\in\N$ and let $(\Xh_n^{a}(\ell/n))_{\ell=0,\dots,n}$ denote
the corresponding Euler scheme with time step $1/n$, i.e.,
$\Xh_n^{a}(0) = 0$ and for $\ell = 1,\dots,n$,
\[
\Xh_n^{a}(\tfrac{\ell}{n}) = \Xh_n^{a}(\tfrac{\ell-1}{n}) + a(\Xh_n^{a}(\tfrac{\ell-1}{n}))\cdot \tfrac{1}{n}.
\]

Since $\|a\|_\infty\leq 1$ we have $|X^{a}(t)|\leq 1$ for all
$t\in[0,1]$ and $|\Xh_n^{a}(\tfrac{\ell}{n})|\leq 1$ for all $\ell =
0,\dots,n$. Using the latter two facts it is straightforward to see that for all
Lipschitz continuous functions $f\colon \R^d\to\R$ we have
\[
|f(X^{a}(1))-f(\Xh_n^{a}(1))|\leq c_1\cdot \|f\|_{\text{Lip}}\cdot
\exp\bigl(c_2\cdot \max_{\alpha\in \N_0^d, |\alpha|_1 =
1}\sup_{|x|\leq 1}|D^\alpha a(x)|\bigr)\cdot \frac{1}{n},
\]
where $c_1$ and $c_2$ are positive reals, which only depend on the
dimension $d$.
\end{rem}

\subsection{Worst case analysis with respect to integrands}\label{sub4-2}
In this section we take $d=m=1$ and we study the quadrature problem
for the trivial equation
\[
\begin{aligned}
dX(t)& = dW(t),\qquad t\in [0,1],\\ X(0)&  = 0.
\end{aligned}
\]
The class of equations is thus given by the singleton
\[
\eq = \{(0,1)\}.
\]

For a function $u\colon [0,\infty)\to [0,\infty)$ we consider the class of integrands given by
\[
 \cF_u=\bigl\{f\in C^{\infty,0}(\R; \R)\colon \|f'/(1+u(|\cdot|))\|_\infty \le 1\bigr\}.
\]
We thus study the computation of $\EE[f(W(1))]$ for all functions $f\colon\R\to\R$ that are infinitely often differentiable, bounded by one and have a derivative $f'$, which satisfies the growth condition specified by $u$.

The following result provides lower bounds for the minimal errors $e_n^{\text{det}}(\ccF(\{(0,1)\}\times\cF_u); \sSDE )$ of deterministic algorithms in case that the function $u$ is continuous, strictly increasing and unbounded.
See Section~\ref{proof2} for a proof.

\begin{theorem}\label{mainThm2}
There exists $c\in (0,\infty)$ such that for all continuous and strictly
increasing functions $u\colon [0,\infty)\to[0,\infty)$ with  $\displaystyle{ \lim_{x\to\infty} u(x) =\infty}$ and every $n\in\N$  we
have
\[
e_n^{\text{det}}(\ccF(\{(0,1)\}\times\cF_u); \sSDE)\geq c \cdot
\exp\bigl(-(u^{-1}( \max(n,u(0)))^2\bigr).
\]
\end{theorem}

By Lemma 5 in \cite{JMGY15},
\[
\lim_{n\to\infty} n^q\cdot \exp\bigl(-(u^{-1}( n))^2\bigr)=\infty
\]
for all $q>0$ if and only if \eqref{101} holds for all $q>0$. Thus,
Theorem~\ref{mainThm2} implies that, analogously to Corollary \ref{cor1}, we can not have a polynomial rate
of convergence of the
minimal errors $e_n^{\text{det}}(\{(0,1)\}\times\cF_u; \sSDE)$ if
$u$ satisfies the exponential growth condition~\eqref{101}.
\begin{cor}
Assume that $u\colon [0,\infty)\to [0,\infty)$ is
          continuous, strictly increasing and
 satisfies \eqref{101}. Then for all $q>0$,
\[
\lim_{n\to\infty} n^q\cdot e_n^{\text{det}}(\{(0,1)\}\times\cF_u;
\sSDE)=\infty.
\]
\end{cor}

Similar to Corollary~\ref{cor2} we may also have an arbitrary slow decay
of the minimal errors
$e_n^{\text{det}}(\ccF(\{(0,1)\}\times \cF_u); \sSDE)$.

\begin{cor}\label{corr4}
For every sequence  $(\eps_n)_{n\in\N}\subset (0, \infty)$ with
$\lim_{n\to\infty} \eps_n=0$ there exists  $c\in (0,\infty)$ and a
strictly increasing and continous function $u\colon [0,\infty) \to [0, \infty)$
such that for all $n\in\N$ we have
\[
e_n^{\text{det}}(\{(0,1)\}\times\cF_u; \sSDE)\geq c\cdot \eps_n.
\]
\end{cor}

\begin{proof}
Without loss of generality we may assume that the sequence $( \eps_n )_{ n \in \N }$
is strictly decreasing.

Choose $c\in (0,\infty)$ according to Theorem~\ref{mainThm2} and
put
\[
  b_n = \Bigl( \ln \frac{1}{\eps_n }
  \Bigr)^{1/2}
\]
for $n\in \N$.
Note that
$
  ( b_n )_{ n\geq n_0 }
$
is strictly increasing and satisfies
$
  \lim_{  n \to \infty } b_n = \infty
$.

Next, define $u \colon [0, \infty) \to [0,\infty)$ recursively by
\begin{align*}
  u(x) =
  \begin{cases}
    1+x
    ,
  &
    \text{if }
    x \leq  b_1
    ,
\\ u(b_n)+(x-b_n)\cdot \max\bigl(1,\tfrac{n+1-u(b_n)}{b_{n+1}-b_n}\bigr)
  ,
  &
  \text{if }
  x \in ( b_n , b_{n+1} ] \text{ and } n \geq  1
  .
\end{cases}
\end{align*}
Then
$ u $ is continuous, strictly increasing and for all $n\geq 1$ we have
\begin{equation}\label{assu2x}
  u(b_n)\geq  n.
\end{equation}

Note that $u(0)=1$. Applying Theorem~\ref{mainThm2} and
observing~\eqref{assu2x} as well as the fact that $u^{-1}$ is
increasing we thus obtain for $n\ge 1$ that
\begin{align*}
e_n^{\text{det}}\bigl(\ccF(\{(0,1)\}\times \cF_u); \sSDE \bigr)  \ge
c\cdot \exp\bigl(-(u^{-1}( n))^2\bigr)\ge c\cdot \exp (- b_n^2) = c \cdot \eps_n,
\end{align*}
which finishes the proof.
\end{proof}

\section{Proofs}\label{proof}

We prove Theorem~\ref{mainThm} in Section~\ref{proof1} and Theorem~\ref{mainThm2} in
Section~\ref{proof2}.

\subsection{Proof of Theorem \ref{mainThm}}\label{proof1}
Throughout the following let
$u\colon [0, \infty)\to [0,\infty)$ be strictly increasing, continuous
and satisfy the condition $\liminf_{x\to\infty} u(x)/x>0$.

In order to construct unfavourable equations~\eqref{sde}
in the class $\eq_u$ we employ a particular construction of functions in $C^{\infty,0}(\R;\R)$ and $C^{\infty,1}(\R;\R)$, which is established in the following two lemmas.

\begin{lemma}\label{pr1}
Let $\tau,\tau_1,\tau_2,v,v_1,v_2\in\R$ with  $v\in (0,\infty)$, $\tau_1<\tau_2$, $v_1<v_2$, and consider the functions
\[
 \eta_{\tau, v}\colon \R\to(0, v],\quad x\mapsto
\begin{cases}
v\cdot\bigl(1-\exp(\tfrac{1}{x-\tau})\bigr),& \text{if }x<\tau,\\
v,& \text{if }x \geq \tau
\end{cases}
\]
and
\[
\thet_{\tau_1, \tau_2, v_1, v_2}\colon \R\to[v_1, v_2],\quad x\mapsto
\begin{cases}
v_1, & \text{if }x\leq \tau_1,\\
v_1+\frac{v_2-v_1}{1+\exp\bigl(\frac{\tau_2-\tau_1}{x-\tau_1}-\frac{\tau_2-\tau_1}{\tau_2-x}\bigr)},& \text{if }x\in(\tau_1, \tau_2),\\
v_2,& \text{if }x \geq \tau_2.
\end{cases}
\]
Then\\[-.4cm]
\begin{itemize}
\item[(i)] $\eta_{\tau, v}\in C^\infty(\R;(0,v])$,\, $\thet_{\tau_1, \tau_2, v_1, v_2}\in C^\infty(\R;[v_1,v_2])$,\\[-.3cm]
\item[(ii)] $\forall k\in\N\colon\, [\tau,\infty)\subset \{\eta_{\tau,v}^{(k)} =0\}$, $\R\setminus (\tau_1,\tau_2)\subset\{\theta_{\tau_1,\tau_2,v_1,v_2}^{(k)} =0\}$,\\[-.3cm]
\item[(iii)] $\eta_{\tau, v}$ is strictly increasing on $(-\infty,\tau]$, \, $\thet_{\tau_1, \tau_2, v_1, v_2}$ is strictly increasing on $[\tau_1,\tau_2]$,\\[-.3cm]
 \item[(iv)] $\|\eta'_{\tau, v}\|_\infty \le 4\exp(-2)\cdot v$,\,
 $\|\thet'_{\tau_1, \tau_2, v_1, v_2}\|_\infty \le 4\cdot \tfrac{v_2-v_1}{\tau_2-\tau_1}$.
\end{itemize}
\end{lemma}

\begin{proof}
Property (i) is well-known and Properties (ii) and (iii) are obvious. In order to prove Property (iv) we first note that
\begin{equation}\label{e234}
\max_{y>0} (y^2\cdot \exp(-y)) = 4 \cdot \exp(-2),\quad \min_{y\in (0,1)} (\exp(-\tfrac{1}{y}) + \exp(-\tfrac{1}{1-y})) = 2\cdot \exp(-2).
\end{equation}
For $x\in (-\infty,\tau)$ we have
\[
\eta'_{\tau, v}(x)=\frac{v}{(x-\tau)^2}\cdot
\exp\bigl(\tfrac{1}{x-\tau}\bigr),
\]
which jointly with the first equality in \eqref{e234} yields the first inequality in (iv). Next, define
\[
\theta\colon (0,1)\to\R,\quad x\mapsto \frac{1}{1+\exp(\tfrac{1}{y}-\tfrac{1}{1-y})}.
\]
Then for all $x\in (\tau_1, \tau_2)$ we have
\[
\thet_{\tau_1, \tau_2, v_1, v_2} (x)=v_1+(v_2-v_1)\cdot \thet(\tfrac{x-\tau_1}{\tau_2-\tau_1}),
\]
and therefore
\begin{equation}\label{pr00}
\thet'_{\tau_1, \tau_2, v_1, v_2} (x)=\frac{v_2-v_1}{\tau_2-\tau_1}\cdot \thet'(\tfrac{x-\tau_1}{\tau_2-\tau_1}).
\end{equation}
For all $y\in(0,1)$ we have
\begin{align*}
 \thet'(y) & =
\Bigl(\frac{1}{y^2}+\frac{1}{(1-y)^2}\Bigr)\cdot
\frac{1}{\bigl(1+\exp(\tfrac{1}{y}-\tfrac{1}{1-y})\bigr)^2}\cdot \exp\bigl(\tfrac{1}{y}-\tfrac{1}{1-y}\bigr)\\
& \le \frac{1}{y^2} \cdot \frac{1}{1+\exp(\tfrac{1}{y}-\tfrac{1}{1-y})} + \frac{1}{(1-y)^2}\cdot
\frac{\exp\bigl(\tfrac{1}{y}-\tfrac{1}{1-y}\bigr)}{1+\exp(\tfrac{1}{y}-\tfrac{1}{1-y})}\\
& = \Bigl( \frac{1}{y^2} \cdot \exp\bigl(-\tfrac{1}{y}\bigr) + \frac{1}{(1-y)^2}\exp\bigl(-\tfrac{1}{1-y}\bigr)\Bigr)\cdot \frac{1}{\exp(-\tfrac{1}{y}) + \exp(-\tfrac{1}{1-y})},
\end{align*}
which jointly with  \eqref{e234} and \eqref{pr00} yields the second inequality in (iv).
\end{proof}

\begin{lemma}\label{pr2}
Consider the functions
\[
\rho_1 = 1/8- \thet_{0,1/2,0,1/8},\,\,\,\,\,\rho_2 = \thet_{1/2,1,0,1/8}
\]
as well as the function
\[
\rho_3\colon \R\to \R,\quad x \mapsto x\cdot \exp(-x^2),
\]
and put
\[
c_{\rho_1} = \int_0^{\tfrac{1}{2}} (\rho_1(x))^2\, dx,\,\, c_{\rho_2} = \int_{\tfrac{1}{2}}^1 \rho_2(x)\, dx.
\]
Then
\begin{itemize}
\item[(i)] $\rho_1,\rho_2,\rho_3\in C^{\infty,1}(\R;\R)$,\\[-.4cm]
\item[(ii)] $ \{\rho_1=0\}= [1/2,\infty) $,\, $ \{\rho_2=0\}=(-\infty,1/2]$,\\[-.4cm]
\item[(iii)] $\forall x\in[0,\infty)\colon 0\le \rho_3(x)= -\rho_3(-x)$,\\[-.4cm]
\item[(iv)] $\forall x\in[1/2,1]\colon \rho_3(x) \ge \rho_3(1)/2 >0$,\\[-.4cm]
\item[(v)] $c_{\rho_1}\in [1/1024,1/128]$,\,\, $c_{\rho_2}\in [1/64,\infty)$.
\end{itemize}
\end{lemma}

\begin{proof}
Property (ii) is an immediate consequence of the definition of $\rho_1$ and $\rho_2$, see Lemma \ref{pr1}, and Property (iii) is obvious.

By Lemma \ref{pr1}(i) we have $\rho_1,\rho_2\in C^{\infty,0}(\R;\R)$.
    By Lemma \ref{pr1}(iii) we have $\|\rho_i'\|_\infty \le 4 \cdot (1/8)/(1/2) = 1$ for
  $i=1,2$.
    Hence $\rho_1,\rho_2\in C^{\infty,1}(\R;\R)$. Clearly, $\rho_3\in  C^{\infty}(\R;\R)$.
 Furthermore,
    $\rho_3'(x) = (1-2x^2)\exp(-x^2)$ for all $x\in\R$, which yields
    $\|\rho_3\|_\infty = \rho_3(1/\sqrt{2}) \le 1$.
    For every $x\in [-1,1]$ we have $|1-2x^2| \le 1 \le \exp(x^2)$,
    while for every $x\in\R$ with $|x| >1$ we get
    \[
    |1-2x^2|  = 2x^2-1 =  1 + \int_1^{x^2} 2\, dy  \le  1 + \int_1^{x^2} \exp(y)\, dy   =  1+ \exp(x^2) - \exp(1) \le \exp(x^2).
    \]
    Hence $\|\rho_3'\|_\infty \le 1$ and we conclude that $\rho_3\in C^{\infty,1}(\R;\R)$.
 Thus Property (i) is proven.

For every $x\in [1/2,1]$ we have $x\cdot \exp(-x^2) \ge \tfrac{1}{2}\cdot \exp(-1) = \tfrac{1}{2}\cdot \rho_3(1)$, which proves Property (iv) in Lemma \ref{pr2}.

Since $0\le \rho_1 \le 1/8$ we get
\[
\int_0^{\tfrac{1}{2}} (\rho_1(x))^2\, dx \le \frac{1}{128}.
\]
Since $\rho_1$ is decreasing on [0,1/2] we have
\[
\int_0^{\tfrac{1}{2}} (\rho_1(x))^2\, dx \ge \int_0^{\tfrac{1}{4}} (\rho_1(x))^2\, dx \ge \frac{1}{4}\cdot (\rho_1(\tfrac{1}{4}))^2 = \frac{1}{4}\cdot \Bigl(\frac{1}{16}\Bigr)^2 =
\frac{1}{1024}.
\]
Furthermore, since $\rho_2$ is increasing on [1/2,1] we get
\[
\int_{\tfrac{1}{2}}^1 \rho_2(x)\, dx \ge \int_{\tfrac{3}{4}}^1 \rho_2(x)\, dx \ge \frac{1}{4}\cdot\rho_2(\tfrac{3}{4})
= \frac{1}{4}\cdot \frac{1}{16} = \frac{1}{64},
\]
which yields Property (v) and completes the proof of the lemma.
 \end{proof}

Using the functions $\rho_1,\rho_2,\rho_3$ from Lemma~\ref{pr3} we construct a subclass of unfavourable equations $\eq_u^*\subset \eq_u$.
Put
\begin{align*}
\Hc&=\bigl{\{}h\in C^{\infty,1}(\R; \R)\colon  \{h\neq 0\}
\subset [0, 1/2]\bigr{\}}, \\
\V_u&=\bigl{\{}v\in C^{\infty}(\R; \R)\colon |v|,|v'|\leq 1+u(|\cdot|)\bigr{\}}.
\end{align*}
We use a single diffusion coefficient $b$ given by
 \begin{equation}\label{diff}
b\colon \R^4\to \R^4,\quad (x_1,\dots,x_4)\mapsto (0, \rho_1(x_4),0,0),
\end{equation}
and for all $h\in \Hc$ and $v\in \V_u$ we define
   a drift coefficient $a^{h, v}$ by
\begin{equation}\label{drift}
a^{h, v}\colon \R^4 \to \R^4,\quad (x_1,\dots,x_4) \mapsto \bigl(\rho_2(x_4)\cdot \rho_3\bigl(\tfrac{x_3}{1+x_3^2}\cdot v(x_2) \bigr), 0, h(x_4),1\bigr).
\end{equation}
Put
\[
\eq_u^* = \{(a^{h, v},b)\colon h\in \Hc,\,v\in \V_u\}.
\]

\begin{lemma}\label{pr3}
We have $\eq_u^*\subset \eq_u$ and, consequently, for every $n\in\N$,
\[
e_n^{\text{ran}}\bigl(\ccF(\eq_u\times\{\pi_1\}); \sSDE\bigr)\geq e_n^{\text{ran}}\bigl(\ccF(\eq_u^*\times\{\pi_1\}); \sSDE\bigr).
\]
\end{lemma}

\begin{proof}
From $\rho_1\in C^{\infty,1}(\R;\R)$ we immediately get $b\in C^{\infty}(\R^4;\R^4)$
as well as
\[
\max_{\alpha\in \N_0^4\colon \alpha_1+\dots+\alpha_4\le 1}\|D^\alpha b\|_\infty= \max(\|\rho_1\|_\infty,\|\rho_1'\|_\infty) \le 1.
\]
From $\rho_2, \rho_3, h,v\in C^{\infty}(\R;\R)$ it is clear that $a^{h, v}\in C^{\infty}(\R^4;\R^4)$.
Furthermore, we have $\|a^{h, v}\|_\infty \le \max(\|\rho_2\|_\infty\cdot \|\rho_3\|_\infty,
\|h\|_\infty,1) = 1$ since $\rho_2, \rho_3, h \in C^{\infty,1}(\R;\R)$.
For every $x=(x_1,\dots,x_4)\in\R^4$ we have
\[
D^{\alpha}a^{h, v}(x) = \begin{cases}
0, & \text{if }\alpha = (1,0,0,0),\\
(\rho_2(x_4)\cdot \tfrac{x_3}{1+x_3^2}\cdot v'(x_2) \cdot\rho_3'\bigl(\tfrac{x_3}{1+x_3^2}\cdot v(x_2) \bigr),0,0,0), & \text{if }\alpha = (0,1,0,0),\\
(\rho_2(x_4)\cdot \tfrac{1-x_3^2}{(1+x_3^2)^2}\cdot v(x_2) \cdot\rho_3'\bigl(\tfrac{x_3}{1+x_3^2}\cdot v(x_2) \bigr),0,0,0), & \text{if }\alpha = (0,0,1,0),\\
(\rho_2'(x_4)\cdot \rho_3\bigl(\tfrac{x_3}{1+x_3^2}\cdot v(x_2) \bigr), 0,
h'(x_4),0),
& \text{if }\alpha = (0,0,0,1),
\end{cases}
\]
and therefore
\begin{align*}
|D^{\alpha}a^{h, v}(x)| & \le  \begin{cases}
0, & \text{if }\alpha = (1,0,0,0),\\
|v'(x_2)|, & \text{if }\alpha = (0,1,0,0),\\
|v(x_2)|, & \text{if }\alpha = (0,0,1,0),\\
 1, & \text{if }\alpha = (0,0,0,1)\end{cases} \\
 & \le 1+ u(|x_2|) \le 1+ u(|x|)
\end{align*}
since $\rho_2,\rho_3,h\in C^{\infty,1}(\R;\R)$, $v\in \V_u$ and $u$ is increasing.

Hence $(a^{h, v},b)\in \eq_u$, which finishes the proof.
\end{proof}

Next, we determine the values of $\sSDE$ on  $\ccF(\eq_u^*,\{\pi_1\})$.

\begin{lemma}\label{pr3aa}
For every $h\in \Hc$ and every $v\in \V_u$ the solution $X^{a^{h,v},b}$ of \eqref{sde} with the drift coefficient $a=a^{h,v}$ given by \eqref{drift} and
   the diffusion coefficient $b$ given by \eqref{diff} satisfies
\[
\EEE[X^{a^{h, v},b}(1)] = \frac{c_{\rho_2}}{\sqrt{2\pi
c_{\rho_1}}}\cdot\int_\R \rho_3\Bigl(\tfrac{\int_0^{1/2} h (t)\,
dt}{1+(\int_0^{\1/2} h (t)\, dt)^2}\cdot v(x)\Bigr)\cdot
\exp(-\tfrac{x^2}{2 c_{\rho_1}})\,dx.
\]
\end{lemma}

\begin{proof}
Let $h\in \Hc$ and $v\in \V_u$, and write $X$ in place of $X^{a^{h,v},b}$. By definition of $a^{h,v}$ and $b$ we have
\begin{align*}
X_1(t)&=\1_{[1/2, 1]}(t)\cdot \rho_3\Bigl(\tfrac{X_3(1/2)}{1+(X_3(1/2))^2}\cdot v(X_2(1/2))\Bigr)\cdot \int_{1/2}^t \rho_2(s)ds,\\
X_2(t)&=\int_0^{\min(t, 1/2)}\rho_1(s)dW(s),\\
X_3(t)&=\int_0^{\min(t, 1/2)}h(s)ds,\\
X_4(t)&=t
\end{align*}
for all $t\in [0,1]$, and therefore
\[
X_1(1) = c_{\rho_2}\cdot \rho_3\Bigl(\tfrac{\int_0^{1/2} h (t)\, dt}{1+(\int_0^{\1/2} h (t)\, dt)^2}\cdot v(X_2(1/2))\Bigr).
\]
Furthermore,
\[
X_2(1/2) = \int_0^{1/2}\rho_1(s)\, dW(s) \,\sim\, \text{N}(0,c_{\rho_1}),
\]
such that
\[
 \EEE(X_1(1)) = \frac{c_{\rho_2}}{\sqrt{2\pi c_{\rho_1}}}\cdot\int_\R \rho_3\Bigl(\tfrac{\int_0^{1/2} h (t)\, dt}{1+(\int_0^{\1/2} h (t)\, dt)^2}\cdot v(x)\Bigr)\cdot \exp(-\tfrac{x^2}{2 c_{\rho_1}})\,dx
\]
as claimed.
\end{proof}

Motivated by Lemma~\ref{pr3aa} we consider the following family of non-linear integration problems in the setting of Section \ref{Sec3}. Take
\[
A=\R, \quad B = \R^{\N_0},\quad \ccF(\Hc)=\{(h^{(k)})_{k\in \N_0}\colon h\in \Hc\},
\]
and for $v \in \V_u$ define $S_v^{\text{int}}\colon \ccF(\Hc)\to\R$ by
\[
S^{\text{int}}_v\bigl((h^{(k)})_{k\in \N_0}\bigr)=\frac{c_{\rho_2}}{\sqrt{2\pi c_{\rho_1}}}\cdot\int_\R \rho_3\Bigl(\tfrac{\int_0^{1/2} h (t)\, dt}{1+(\int_0^{\1/2} h (t)\, dt)^2}\cdot v(x)\Bigr)\cdot \exp(-\tfrac{x^2}{2 c_{\rho_1}})\,dx.
\]

\begin{lemma}\label{pr3a}
For every $v\in \V_u$ and every $n\in\N$ we have
\[
e_n^{\text{ran}}\bigl(\ccF(\eq_u^*\times\{\pi_1\}); \sSDE\bigr) \ge
e_n^{\text{ran}}\bigl(\ccF(\Hc); S_v^{\text{int}}\bigr).
\]
\end{lemma}

\begin{proof}
Let $v\in \V_u$ and $n\in\N$. We use $\Aran_\text{sde}$ and $\Aran_\text{int}$ to denote the class of randomized algorithms for the approximation of
$\sSDE\colon \ccF(\eq_u^*\times\{\pi_1\})\to\R$
and $S_v^{\text{int}}$, respectively.
In order to prove the lemma we define
$T\colon \ccF(\Hc) \to \ccF(\eq_u^*\times\{\pi_1\})$ by
\[
T((h^{(k)})_{k\in \N_0}) =
(D^\alpha a^{h,v}, D^\alpha b, D^\alpha \pi_1)_{\alpha\in \N_0^4}.
\]
By Lemma~\ref{pr3aa},
\begin{equation}\label{first}
S_v^{\text{int}} = \sSDE \circ T.
\end{equation}
Below we show
\begin{equation}\label{sec}
\forall\,\Sh^{\text{sde}}\in \Aran_\text{sde}\,\,\,\,\exists\,\Sh_v^{\text{int}}\in \Aran_\text{int}\colon\, \cost(\Sh_v^{\text{int}}) \le \cost(\Sh^{\text{sde}})\,\,\text{ and }\,\,  \Sh_v^{\text{int}} = \Sh^{\text{sde}}\circ T.
\end{equation}
Clearly, \eqref{first} and \eqref{sec} jointly imply that for every $\Sh^{\text{sde}}\in \Aran_\text{sde}$ with $\cost(\Sh^{\text{sde}})\le n$ there exists $\Sh_v^{\text{int}}\in \Aran_\text{int}$ such that $\cost(\Sh_v^{\text{int}}) \le n$
and
\begin{align*}
e(\Sh^{\text{sde}}) & = \sup_{g\in \ccF(\eq_u^*\times\{\pi_1\})} \EE|\sSDE(g)- \Sh^{\text{sde}}(g)| \\
&\geq \sup_{g\in \ccF(\{(a^{h, v},b)\colon h\in \Hc\}\times\{\pi_1\})} \EE|\sSDE(g)- \Sh^{\text{sde}}(g)|\\
&= \sup_{g\in \ccF(\Hc)} \EE|\sSDE\circ T(g)- \Sh^{\text{sde}}\circ T(g)| = e(\Sh_v^{\text{int}}).
\end{align*}
Hence
$e_n^{\text{ran}}\bigl(\ccF(\eq_u^*\times\{\pi_1\}); \sSDE\bigr) \ge
e_n^{\text{ran}}\bigl(\ccF(\Hc); S_v^{\text{int}}\bigr)$.

It remains to prove~\eqref{sec}. To this end we first note that there exists a mapping $\widetilde T\colon \R^4 \times \R^{\N_0}\to (\R^4\times \R^4\times \R)^{\N_0^4}$
 such that $T g(x) = \widetilde T(x,g(x_4))$ for all $g\in \ccF(\Hc)$ and all $x=(x_1,\dots,x_4)\in\R^4$.
Next, let $\Sh^{\text{sde}}\in \Aran_\text{sde}$ be given by a probability space $(\Omega, \A, \PP)$ and a triple $(\psi,\nu,\varphi)$, i.e., $\Sh^{\text{sde}} = \Sh^{\text{sde}}_{\psi,\nu,\varphi}$ with mappings
\[
\begin{aligned}
\psi_k\colon & \bigl((\R^4\times \R^4 \times \R)^{\N_0^4}\bigr)^{k-1}\times \Omega \to \R^4,\quad k\in\N,\\
\nu\colon & \ccF(\eq_u^*\times\{\pi_1\}) \times \Omega \to \N,\\
\varphi_k\colon & \bigl((\R^4\times \R^4 \times \R)^{\N_0^4}\bigr)^{k}\times \Omega \to \R,\quad k\in\N,
\end{aligned}
\]
see Section \ref{Sec3}. Let $\pi_4\colon \R^4\to \R$ denote the
projection on the fourth component. We define mappings
\[
\begin{aligned}
\widetilde \psi_k\colon & \bigl(\R^{\N_0}\bigr)^{k-1}\times \Omega \to \R,\quad k\in\N,\\
\tilde \nu\colon & \ccF(\Hc) \times \Omega \to \N,\\
\widetilde\varphi_k\colon & \bigl(\R^{\N_0}\bigr)^{k}\times \Omega \to \R,\quad k\in\N,
\end{aligned}
\]
by taking $\widetilde \psi_1=\pi_4(\psi_1)$ and
\begin{align*}
\widetilde \psi_k(\tilde y_1,\dots,\tilde y_{k-1},\omega) & = \pi_4(\psi_k(y_1,\dots,y_{k-1},\omega)),\quad k\geq 2\\
\tilde \nu (g,\omega) & = \nu (Tg,\omega),\\
\widetilde\varphi_k(\tilde y_1,\dots,\tilde y_{k},\omega)& = \varphi_k(y_1,\dots, y_{k},\omega),\quad k\in\N,
\end{align*}
where $y_1 = \widetilde T(\psi_1(\omega), \tilde y_1), y_2= \widetilde T(\psi_2(y_1,\omega), \tilde y_2),\dots, y_k=\widetilde T(\psi_k(y_1,\dots,y_{k-1},\omega), \tilde y_k)$. Put $\widetilde \psi =(\widetilde \psi_k)_{k\in\N}$ and $\widetilde \varphi =(\widetilde \varphi_k)_{k\in\N}$. Then
$\Sh^{\text{int}}_{v,(\tilde \psi,\tilde \nu,\tilde \varphi)} \in \Aran_\text{int}$ and by the definition of the mappings $\widetilde \psi_k$, $\tilde \nu$ and $\widetilde \varphi_k$ we have $\Sh^{\text{int}}_{v,(\tilde \psi,\tilde \nu,\tilde \varphi)}(g) =  \Sh^{\text{sde}}_{\psi,\nu,\varphi}(Tg)$ for all $g\in \ccF(\Hc)$. Moreover,
\[
\sup_{g\in \ccF(\Hc)}\EE \tilde \nu (g,\cdot)  = \sup_{g\in \ccF(\Hc)}  \EE \nu (Tg,\cdot) \leq  \sup_{g\in \ccF(\eq_u^*\times\{\pi_1\})}  \EE \nu (g,\cdot),
\]
which shows that $\cost(\Sh_{v,(\tilde \psi,\tilde \nu,\tilde \varphi)}^{\text{int}}) \le \cost(\Sh^{\text{sde}}_{\psi,\nu,\varphi})$ and thus finishes the proof of \eqref{sec}.
\end{proof}

In view of Lemma~\ref{pr3a} it suffices to establish appropriate lower bounds for the minimal errors $e_n^{\text{ran}}\bigl(\ccF(\Hc); S_v^{\text{int}}\bigr)$. To this end we first construct unfavourable functions $h\in \Hc$ in Lemma~\ref{pr4} and then employ Proposition~\ref{prop1} in Lemma~\ref{pr5}.

\begin{lemma}\label{pr4}
For every $m\in\N$ there exist $h_1,\dots,h_{m}\in \Hc$ with the following properties.
\begin{itemize}
\item[(i)] The sets $\{h_{1}\neq 0\}$, \ldots, $\{h_{m}\neq 0\}$
are pairwise disjoint.\\[-.3cm]
\item[(ii)] For all $i\in\{1,\ldots, m\}$ we have
\[
\int_0^{\tfrac{1}{2}}h_{i}(t)\,dt=\frac{1}{(12m)^2}.
\]
\end{itemize}
\end{lemma}

\begin{proof}
Let
\[
h\colon \R\to [0,1/(6m)],\quad x\mapsto
\begin{cases}
\thet_{0,\tfrac{1}{6m},0,\tfrac{1}{6m} }(x), & \text{if }x\leq \tfrac{2}{6m},\\[.3cm]
\tfrac{1}{6m}-\thet_{\tfrac{2}{6m},\tfrac{3}{6m},0,\tfrac{1}{6m}
}(x),& \text{if }x>\tfrac{2}{6m},
\end{cases}
\]
and put
\[
c_0 = \int_0^{\tfrac{3}{6m}} h(x)\, dx.
\]
For $i= 1, \dots, m $ we define
\[
h_{i}\colon \R\to \R, \quad x\mapsto \frac{1}{c_0\cdot (12m)^2}\cdot h \Bigl(x-\tfrac{i-1}{2m}\Bigr).
\]

Let $i\in \{1,\dots,m\}$.
From Lemma \ref{pr1}(i) we get $h\in C^\infty(\R;\R)$ and therefore we have $h_i\in C^\infty(\R;\R)$. By the definition of $h$ we have $h(x) = 1/(6m)$ for all  $x\in [ 1/(6m),2/(6m)]$, which implies
\[
c_0 \ge \int_{\tfrac{1}{6m}}^{\tfrac{2}{6m}} h(x)\, dx = \frac{1}{(6m)^2}.
\]
It follows that
\[
\|h_i\|_\infty = \frac{1}{c_0\cdot (12m)^2}\cdot \|h\|_\infty \leq
\frac{1}{4\cdot 6m} \le 1
\]
and, using Lemma \ref{pr1}(iv),
\[
\|h_i'\|_\infty = \frac{1}{c_0\cdot (12m)^2}\cdot \|h'\|_\infty \le \frac{ \|h'\|_\infty}{4} = 1.
\]
Hence $h_i\in C^{\infty,1}(\R;\R)$.

By the definition of $h$ we have $\{h\neq 0\} = (0,3/(6m))$, which implies
\[
\{h_i\neq 0\} = \bigl(\tfrac{i-1}{2m},\tfrac{i}{2m}\bigr) \subset [0,1/2].
\]
Thus, $h_i\in \Hc$ and statement (i) of the lemma holds.

Finally,
\[
\int_0^{\tfrac{1}{2}}h_{i}(x)\,dx = \frac{1}{c_0\cdot (12m)^2}\cdot \int_0^{\tfrac{3}{6m}} h(x)\, dx = \frac{1}{(12m)^2},
\]
which proves statement (ii) of the lemma and finishes the proof.
\end{proof}

\begin{lemma}\label{pr5}
For every $v\in \V_u$ and every $n\in\N$ we have
\[
e_n^{\text{ran}}\bigl(\ccF(\Hc); S_v^{\text{int}}\bigr) \ge \frac{
c_{\rho_2}}{68\sqrt{2\pi c_{\rho_1}}} \cdot\int_\R
\rho_3\Bigl(\tfrac{4\cdot (102 n)^2} {1+ 16\cdot (102 n)^4}\cdot
v(x)\Bigr)\cdot \exp\bigl(-\tfrac{x^2}{2 c_{\rho_1}}\bigr)\,dx.
\]
\end{lemma}

\begin{proof}
Let $v\in \V_u$ and $n\in\N$ and choose $h_1,\dots,h_{17n}\in \Hc$ according to Lemma \ref{pr4} with $m=17n$. By Lemma \ref{pr4}(ii)  and the fact that $\rho_3(t) = -\rho_3(-t)$ for all $t\ge 0$ we obtain
\begin{align*}
& S^{\text{int}}_v((h_i^{(k)})_{k\in\N_0})-S_v^{\text{int}}(-(h_i^{(k)})_{k\in\N_0})\\ & \qquad\qquad\qquad = \frac{2c_{\rho_2}}{\sqrt{2\pi c_{\rho_1}}}\cdot\int_\R \rho_3\Bigl(\tfrac{4\cdot (102 n)^2}{1+ 16\cdot (102 n)^4}\cdot v(x)\Bigr)\cdot \exp\bigl(-\tfrac{x^2}{2 c_{\rho_1}}\bigr)\,dx
\end{align*}
for $i=1,\dots,17n$. It remains to apply Proposition \ref{prop1} with $A=\R$, $B=\R^{\N_0}$, $\ccF=\ccF(\Hc)$, $m=17n$, $g_{i,+}=(h_i^{(k)})_{k\in\N_0}$,  $g_{i,-}=(-h_i^{(k)})_{k\in\N_0}$ and $b^*=0\in \R^{\N_0}$.
\end{proof}

The following lemma provides a technical tool for the construction
of unfavorable functions $v_n\in \V_u$ carried out in
Lemma~\ref{pr7}.

\begin{lemma}\label{pr6}
Let $\delta >0$ and let $\xl\in (0,\infty)$ satisfy $\inf_{x\ge \xl} u(x)/x \ge \delta$.
Then the function
\begin{equation}\label{func}
\vl\colon [1,\infty)\to [0,\infty),\quad x\mapsto \Bigl(\int_0^{x-1} u(y)\, dy\Bigr)^{1/2}
\end{equation}
satisfies
 \\[-.4cm]
\begin{itemize}
\item[(i)] $\vl$ is strictly increasing, differentiable on $(1,\infty)$
and $\vl([1,\infty)) = [0,\infty)$,\\[-.3cm]
\item[(ii)] $\forall\, x\in [1,\infty)\colon \, \vl(x)\le \frac{1}{\sqrt{\delta}}\,u(\max(x-1,x_\delta))$,\\[-.3cm]
\item[(iii)] $\forall\, x\in [x_\delta+1,\infty)\colon \, \vl'(x)\ge \sqrt{\delta}/2$,\\[-.3cm]
\item[(iv)] $\forall\, y\in [\vl(2),\infty)\colon\, \vl^{-1}(y) \le 4u^{-1}(y^2)$.
\end{itemize}
\end{lemma}

\begin{proof}
Since $u$ is continuous and positive, the mapping $\vl$ is strictly increasing
 on $[1,\infty)$
and continuously differentiable
    on $(1,\infty)$
with
\[
\vl'(x) = \frac{u(x-1)}{2}\cdot \Bigl(\int_0^{x-1} u(y)\, dy\Bigr)^{-1/2}
\]
for every $x\in (1,\infty)$. Clearly, $\vl(1) = 0$. Moreover, for
$x\in [x_\delta+2,\infty)$ we have $\vl^2(x) \ge \int_{x-2}^{x-1}
u(y)\,dy \ge u(x-2)\ge \delta\cdot (x-2)$, which implies
$\lim_{x\to\infty} \vl(x)=\infty$ and completes the proof of (i).

For $x\in [1,\infty)$ we have
\begin{align*}
\vl^2(x)  & \le \int_0^{\max(x-1,x_\delta)} u(y)\, dy \le \max(x-1,x_\delta)\cdot u(\max(x-1,x_\delta))\\
 & \le \delta^{-1}\cdot u^2(\max(x-1,x_\delta)),
\end{align*}
which implies (ii).

Next, let $x\in [x_\delta+1,\infty)$. Then $x-1\ge x_\delta$, and therefore
 $u(x-1)\ge \delta\cdot (x-1)$. It follows
\[
\vl'(x) \ge \frac{u(x-1)}{2}\cdot \bigl((x-1)\cdot u(x-1)\bigr)^{-1/2} = \frac{1}{2}\Bigl(\frac{u(x-1)}{x-1}\Bigr)^{1/2} \ge \frac{\sqrt{\delta}}{2},
\]
which proves (iii).

Finally, let $y\in [\vl(2),\infty)$. Then $\vl^{-1}(y)\ge 2$, which implies $\vl^{-1}(y)-2 \ge \vl^{-1}(y)/4$. It follows
\begin{align*}
y^2 = (\vl(\vl^{-1}(y)))^2 \ge \int^{\vl^{-1}(y)-1}_{\vl^{-1}(y)-2} u(x)\, dx\ge u( \vl^{-1}(y)-2)\ge u( \vl^{-1}(y)/4),
\end{align*}
which shows (iv) and finishes the proof of the lemma.
\end{proof}

\begin{lemma}\label{pr7}
Let $\delta >0$, let $\xl\in (0,\infty)$ satisfy $\inf_{x\ge \xl} u(x)/x \ge \delta$, let
\[
\kappa_\delta = \frac{8}{\sqrt{\delta}}\cdot \max(u(x_\delta+1),1)
\]
and let $\vl\colon [1,\infty)\to [0,\infty)$ be the function given by \eqref{func}. Let $n\in\N$ and let
\[
\alpha_n = \vl^{-1}\Bigl(\kappa_\delta\cdot \tfrac{1+ 16\cdot (102 n)^4}{4\cdot (102 n)^2}\Bigr).
\]
Then there exists a function
$v_n\in C^\infty(\R;(0,\infty))$ such that\\[-.4cm]
 \begin{itemize}
\item[(i)] $v_n$ is strictly increasing,\\[-.2cm] 
\item[(ii)] $v_n,v_n'\le 1 + u(|\cdot|)$,\\[-.2cm]
\item[(iii)] $\int_0^{\alpha_n} \rho_3\bigl(\tfrac{4\cdot (102 n)^2}{1+ 16\cdot (102 n)^4}\cdot v_n(x)\bigr)\, dx \ge \frac{\rho_3(1)}{4}$.
\end{itemize}
\end{lemma}

\begin{proof}
Note that $\alpha_n$ is well-defined due to Lemma \ref{pr6}(i). We put
\[
\beta_n = \frac{1+ 16\cdot (102 n)^4}{4\cdot (102 n)^2}.
\]

We first show that
\begin{equation}\label{e22}
\alpha_n \ge \xl+2
\end{equation}
and
\begin{equation}\label{e23}
\vl(\alpha_n)-\vl(\alpha_n-1) \ge \frac{\sqrt{\delta}}{2}.
\end{equation}

By Lemma \ref{pr6}(ii) and the fact that $\kappa_\delta\ge u(\xl+1)/\sqrt{\delta}$ we have
\[
\vl(\xl+2)\le \frac{1}{\sqrt{\delta}}u(\xl+1) \le \kappa_\delta\le \kappa_\delta\cdot \beta_n,
\]
which yields \eqref{e22} since $\vl$ is increasing. From \eqref{e22} and Lemma \ref{pr6}(iii) we obtain that $\inf_{x\in [\alpha_n-1,\alpha_n]}\vl'(x) \ge \sqrt{\delta}/2$, which implies \eqref{e23} by the mean value theorem.

Note that $\alpha_n > 2$ due to \eqref{e22} and put
\[
k_n = \max\{k\in\N\colon \alpha_n-k >1\}.
\]
Clearly,
\begin{equation}\label{e24}
\alpha_n-k_n \in (1,2].
\end{equation}

Observing \eqref{e23} and the fact that $\vl$ is increasing we may define a function $v_n\colon \R\to \R$ by
\[
v_n(x)=\frac{1}{\kappa_\delta}\cdot
\begin{cases}
\eta_{\alpha_n-k_n,\vl(\alpha_n-k_n)}(x), & \text{if } x\in (-\infty,\alpha_n-k_n),\\
\theta_{\alpha_n+k,\alpha_n+k+1,\vl(\alpha_n+k),\vl(\alpha_n+k+1)}(x), & \text{if }x\in [\alpha_n+k,\alpha_n+k+1)\\
& \quad \text{and }k\in \{-k_n,\dots,-2\}\cup\N_0,\\
\theta_{\alpha_n-1,\alpha_n-1/2,\vl(\alpha_n-1),\vl(\alpha_n)-\sqrt{\delta}/4}(x), & \text{if }x\in [\alpha_n-1,\alpha_n-1/2),\\
\theta_{\alpha_n-1/2,\alpha_n,\vl(\alpha_n)-\sqrt{\delta}/4,\vl(\alpha_n)}(x), & \text{if }x\in [\alpha_n-1/2,\alpha_n).
\end{cases}
\]

By Lemma \ref{pr1}(i),(ii) we immediately get $v_n\in C^{\infty}(\R;(0,\infty))$. Furthermore, Property (i) is a straightforward consequence of Lemma \ref{pr1}(iii) together with Lemma \ref{pr6}(i).

We turn to the proof of Property (ii). First assume that $x\in (-\infty, \alpha_n-k_n)$.
Using Lemma \ref{pr1}(iii),(iv) as well as \eqref{e24}, the fact that $\vl$ is increasing and Lemma \ref{pr6}(ii) we get
\begin{align*}
v_n(x) + v_n'(x) & = \frac{1}{\kappa_\delta}\cdot \bigl(
\eta_{\alpha_n-k_n,\vl(\alpha_n-k_n)}(x) +\eta_{\alpha_n-k_n,\vl(\alpha_n-k_n)}'(x)\bigr)\\
& \le \frac{1+4\exp(-2)}{\kappa_\delta}\cdot \vl(\alpha_n-k_n)
 \le \frac{2}{\kappa_\delta}\cdot \vl(2) \\
& \le \frac{2}{\kappa_\delta}\cdot \vl(x_\delta +2) \le \frac{2}{\kappa_\delta}\cdot \frac{u(x_\delta +1)}{\sqrt{\delta}} \le \frac{1}{4}\le 1.
\end{align*}
Next, assume that $x\in [\alpha_n+k,\alpha_n+k+1)$ with $k\in\{-k_n,\dots,-2\}\cup \N_0$. Using Lemma \ref{pr1}(iii),(iv) and  the fact that $\vl$ is non-negative, we obtain
\begin{align*}
& v_n(x) + v_n'(x) \\
& \qquad = \frac{\theta_{\alpha_n+k,\alpha_n+k+1,\vl(\alpha_n+k),\vl(\alpha_n+k+1)}(x)+\theta_{\alpha_n+k,\alpha_n+k+1,\vl(\alpha_n+k),\vl(\alpha_n+k+1)}'(x)}{\kappa_\delta}\\
& \qquad \le  \frac{\vl(\alpha_n+k+1) + 4(\vl(\alpha_n+k+1)-\vl(\alpha_n+k))}{\kappa_\delta } \le  \frac{5\vl(\alpha_n+k+1) }{\kappa_\delta}.
\end{align*}
Note that $\alpha_n+k+1\ge \alpha_n-k_n+1 \ge 2$. Using Lemma
\ref{pr6}(ii) and the fact that $u$ is increasing we may therefore
conclude that
\[
\frac{5\cdot \vl(\alpha_n+k+1) }{\kappa_\delta}\le \frac{5\cdot u(\max(\alpha_n+k,x_\delta))}{8 \cdot \max(u(x_\delta+1),1)} \le \max (1,u(\alpha_n+k))\le 1 + u(x).
\]
Next, assume that $x\in [\alpha_n-1,\alpha_n-1/2)$. Using Lemma \ref{pr1}(iii),(iv), the fact that $\vl \ge 0$, \eqref{e22}, Lemma \ref{pr6}(ii) and  the fact that $u$ is increasing we get
\begin{align*}
& \max(v_n(x),v_n'(x)) \\
& \qquad = \frac{\max\bigl(\theta_{\alpha_n-1,\alpha_n-1/2,\vl(\alpha_n-1),\vl(\alpha_n)-\sqrt{\delta}/4}(x),\theta_{\alpha_n-1,\alpha_n-1/2,\vl(\alpha_n-1),\vl(\alpha_n)-\sqrt{\delta}/4}'(x)\bigr)}{\kappa_\delta}\\
& \qquad \le \frac{\max\bigl(\vl(\alpha_n)-\sqrt{\delta}/4, 8(\vl(\alpha_n)-\sqrt{\delta}/4-\vl(\alpha_n-1))\bigr)}{\kappa_\delta}\\
& \qquad \le \frac{8\vl(\alpha_n)}{\kappa_\delta} \le u(\alpha_n-1) \le u(x).
\end{align*}
Finally, assume that $x\in [\alpha_n-1/2,\alpha_n)$. Using Lemma \ref{pr1}(iii),(iv) as well as   \eqref{e23}, Lemma \ref{pr6}(ii), \eqref{e22} and the fact that $u$ is increasing we get
\begin{align*}
&  \max(v_n(x),v_n'(x))\\
& \qquad = \frac{\max\bigl(\theta_{\alpha_n-1/2,\alpha_n,\vl(\alpha_n)-\sqrt{\delta}/4,\vl(\alpha_n)}(x),\theta_{\alpha_n-1/2,\alpha_n,\vl(\alpha_n)-\sqrt{\delta}/4,\vl(\alpha_n)}'(x)\bigr)}{\kappa_\delta}\\
& \qquad \le \frac{\max(\vl(\alpha_n), 2\sqrt{\delta})}{\kappa_\delta }
\le \frac{4\vl(\alpha_n) }{\kappa_\delta}
\le \frac{1}{2}\cdot u(\alpha_n-1) \le u(x).
\end{align*}

It remains to prove Property (iii). Since $\rho_3 \ge 0$ on $[0,\infty)$,  see Lemma \ref{pr2}(iii),
we have
\[
\int_0^{\alpha_n} \rho_3(\tfrac{ v_n(x)}{\beta_n})\, dx \ge \int_{\alpha_n-1/2}^{\alpha_n} \rho_3(\tfrac{ v_n(x)}{\beta_n})\, dx.
\]
Let $x\in [\alpha_n-1/2,\alpha_n]$. Then
\[
\frac{ v_n(x)}{\beta_n}\le  \frac{\vl(\alpha_n)}{\beta_n\cdot \kappa_\delta } = 1.
\]
Furthermore, by Lemma \ref{pr1}(iii),
\begin{align*}
\frac{ v_n(x)}{\beta_n} & \ge  \frac{\vl(\alpha_n)-\sqrt{\delta}/4}{\beta_n\cdot \kappa_\delta } = 1 - \frac{\sqrt{\delta}}{4\beta_n\cdot \kappa_\delta }\ge 1 - \frac{\delta}{32\cdot u(x_\delta+1) }.
\end{align*}
Since $u(x_\delta+1) \ge \delta\cdot (x_\delta +1) \ge \delta$ we conclude that
$v_n(x)/\beta_n \in [1/2,1]$. Therefore,
by Lemma \ref{pr2}(iv),
\[
\rho_3(\tfrac{ v_n(x)}{\beta_n}) \ge \tfrac{\rho_3(1)}{2}.
\]
Hence
\[
 \int_{\alpha_n-1/2}^{\alpha_n}\rho_3(\tfrac{ v_n(x)}{\beta_n}) \, dx \ge \frac{\rho_3(1)}{4},
\]
which completes the proof of the lemma.
\end{proof}

\begin{prop}\label{pr8}
Let $\delta \in (0,\infty)$ and let $\xl\in (0,\infty)$ satisfy $\inf_{x\ge \xl} u(x)/x \ge \delta$ and let
\[
\kappa_\delta = \frac{8}{\sqrt{\delta}}\cdot \max(u(x_\delta+1),1)
\]
Then for every $n\in\N$,
\[
e_n^{\text{ran}}\bigl(\ccF(\eq_u\times\{\pi_1\}); \sSDE \bigr)\geq
\frac{1}{17\cdot 2^7\cdot e\cdot \sqrt{\pi}}
\cdot \exp\bigl(- 2^{13}\cdot
\bigl(u^{-1}(   17^2\cdot 56^4\cdot \kappa_\delta^2\cdot n^4)\bigr)^2\bigr).
\]
\end{prop}

\begin{proof}
Let $\vl\colon [1,\infty)\to [0,\infty)$ be the function given by \eqref{func}, let $n\in\N$ and choose a function $v_n\colon \R\to (0,\infty)$ according to Lemma \ref{pr7}. Note that $v_n\in \V_u$ due to Lemma~\ref{pr7}(ii). Hence, by Lemmas \ref{pr3}, \ref{pr3a}, \ref{pr5}, \ref{pr7} and Lemma~\ref{pr2}
\begin{align*}
e_n^{\text{ran}}\bigl(\ccF(\eq_u\times\{\pi_1\}); \sSDE \bigr) & \geq e_n^{\text{ran}}\bigl(\ccF(\Hc);S_{v_n}^{\text{int}}\bigr)\\
& \ge \frac{c_{\rho_2}}{68\sqrt{2\pi c_{\rho_1}}}
\cdot\int_\R \rho_3\Bigl(\tfrac{4\cdot (102 n)^2}{1+ 16\cdot (102 n)^4}\cdot v_n(x)\Bigr)\cdot \exp\bigl(-\tfrac{x^2}{2 c_{\rho_1}}\bigr)\,dx \\
& \ge \frac{  c_{\rho_2}}{68\sqrt{2\pi
c_{\rho_1}}}\cdot\int_0^{\alpha_n}
 \rho_3\Bigl(\tfrac{4\cdot (102 n)^2}{1+ 16\cdot (102 n)^4}\cdot v_n(x)\Bigr)\cdot \exp\bigl(-\tfrac{x^2}{2 c_{\rho_1}}\bigr)\,dx \\
& \ge \frac{c_{\rho_2}\cdot \rho_3(1)}{272\sqrt{2\pi c_{\rho_1}}}\cdot
 \exp\bigl(-\tfrac{\alpha_n^2}{2 c_{\rho_1}}\bigr)
 \ge \frac{1}{17\cdot 2^7\cdot e\cdot \sqrt{\pi}}
 \cdot \exp(- 2^9\cdot \alpha_n^2).
\end{align*}
Put
\[
\beta_n = \frac{1+ 16\cdot (102 n)^4}{4\cdot (102 n)^2}.
\]
By  \eqref{e22} we have $\alpha_n >2$. Hence $\kappa_\delta\cdot \beta_n \ge \vl(2)$ and we may apply Lemma \ref{pr6}(iv) to obtain
\[
\alpha_n = \vl^{-1}(\kappa_\delta\cdot \beta_n) \le
4u^{-1}(\kappa_\delta^2 \cdot \beta_n^2).\]
Use $\beta_n\le 17/4\cdot (102n)^2$ and the fact that $u^{-1}$ is increasing to complete the proof.
\end{proof}

Clearly, Proposition~\ref{pr8} implies Theorem~\ref{mainThm}.

\subsection{Proof of Theorem \ref{mainThm2}}\label{proof2}
Throughout the following let
$u\colon [0, \infty)\to [0,\infty)$ be strictly increasing,
  continuous
and satisfy $\displaystyle{\lim_{x\to\infty} u(x) = \infty}$.

Put
\[
\Hc_u=\{\hehe \in C^\infty(\R; \R)\colon |\hehe'|\leq 1+u\}.
\]
In the setting of Section~\ref{Sec3} we take
\[
A=\R, \quad B=\R^{\N_0},\quad \ccF(\Hc_u)=\{(h^{(k)})_{k\in\N_0}\colon h\in\Hc_u\},
\]
and we define $S^{\text{int}}\colon \ccF(\Hc_u)\to\R$ by
\[
S^{\text{int}}((h^{(k)})_{k\in\N_0})=\frac{1}{\sqrt{2\pi}}\int_\R \sin(\hehe(x))\cdot
\exp(-\tfrac{x^2}{2})\,dx.
\]

\begin{lemma}\label{pr9}
For every $n\in\N$ we have
\[
e_n^{\text{det}}(\ccF(\{(0,1)\}\times\cF_u); \sSDE)\geq e_n^{\text{det}}(\ccF(\Hc_u);S^{\text{int}}).
\]
\end{lemma}

\begin{proof}
We use $\Adet_\text{sde}$ and $\Adet_\text{int}$ to denote the
classes of deterministic algorithms for the approximation of
$\sSDE\colon \ccF(\{(0,1)\}\times\cF_u)\to\R$ and $S^{\text{int}}$,
respectively. Note that $\sin \circ\hehe \in\cF_u$ for every
$\hehe\in\Hc_u$. We can therefore define a mapping $T\colon
\ccF(\Hc_u)\to \ccF(\{(0,1)\}\times\cF_u)$ by
\[
T((h^{(k)})_{k\in\N_0}) = (0,1_{\{0\}}(k),(\sin\circ h)^{(k)})_{k\in\N_0}.
\]
Clearly,
\[
S^{\text{int}} = \sSDE \circ T.
\]
Similar to the proof of Lemma~\ref{pr3a} it thus remains to show that
\begin{equation}\label{det2}
\forall\,\Sh^{\text{sde}}\in \Adet_\text{sde}\,\,\,\,\exists\,\Sh^{\text{int}}\in \Adet_\text{int}\colon\, \cost(\Sh^{\text{int}}) \le \cost(\Sh^{\text{sde}})\text{\, and\, }  \Sh^{\text{int}} = \Sh^{\text{sde}}\circ T.
\end{equation}

In order to prove~\eqref{det2} we first note that for every $k\in\N$ there exists a mapping $\rho_k\colon \R^{k+1}\to\R$ such that
\[
(\sin\circ h)^{(k)} = \rho_k\circ (h,h^{(1)},\dots,h^{(k)})
\]
for every $k$-times differentiable function $h\colon \R\to\R$, and we define
\[
\rho\colon \R^{\N_0}\to  (\R\times \R\times\R)^{\N_0},\quad (x_k)_{k\in\N_0}\mapsto (0,1_{\{0\}}(k),\rho_k(x_0,\dots,x_k))_{k\in\N_0}.
\]
Next, let $\Sh^{\text{sde}} \in \Adet_\text{sde}$ be given by
$\Sh^{\text{sde}}=\Sh^{\text{sde}}_{\psi,\nu,\varphi}$  with
mappings
\[
\begin{aligned}
\psi_k\colon & \bigl((\R\times \R \times \R)^{\N_0}\bigr)^{k-1} \to \R,\quad k\in\N,\\
\nu\colon & \ccF(\{(0,1)\}\times\cF_u)  \to \N,\\
\varphi_k\colon & \bigl((\R\times \R \times \R)^{\N_0}\bigr)^{k} \to \R,\quad k\in\N,
\end{aligned}
\]
see Section \ref{Sec3}. We define mappings
\[
\begin{aligned}
\widetilde \psi_k\colon & \bigl(\R^{\N_0}\bigr)^{k-1}\to \R,\quad k\in\N,\\
\tilde \nu\colon & \ccF(\Hc_u) \to \N,\\
\widetilde\varphi_k\colon & \bigl(\R^{\N_0}\bigr)^{k} \to \R,\quad k\in\N,
\end{aligned}
\]
by taking $\widetilde \psi_1 = \psi_1$ and
\begin{align*}
\widetilde \psi_k(\tilde y_1,\dots,\tilde y_{k-1}) & = \psi_k(\rho(\tilde y_1),\dots,\rho(\tilde y_{k-1})),\quad k\ge 2,\\
\tilde \nu (g) & = \nu (Tg),\\
\widetilde\varphi_k(\tilde y_1,\dots,\tilde y_{k})& = \varphi(\rho(\tilde y_1),\dots,\rho(\tilde y_{k})),\quad k\in \N.
\end{align*}
Put $\widetilde \psi =(\widetilde \psi_k)_{k\in\N}$ and $\widetilde \varphi =(\widetilde \varphi_k)_{k\in\N}$.
Then
$\Sh^{\text{int}}_{\tilde \psi,\tilde \nu,\tilde \varphi} \in \Adet_\text{int}$ and by the definition of the mappings $\widetilde \psi_k$, $\tilde \nu$ and $\widetilde \varphi_k$ we have $\Sh^{\text{int}}_{\tilde \psi,\tilde \nu,\tilde \varphi}(g) =  \Sh^{\text{sde}}_{\psi,\nu,\varphi}(Tg)$ for all $g\in \ccF(\Hc_u)$. Moreover,
\[
\sup_{g\in \ccF(\Hc_u)}\tilde \nu (g,\cdot)  = \sup_{g\in \ccF(\Hc_u)}  \nu (Tg,\cdot) \le \sup_{g\in \ccF(\{(0,1)\}\times\cF_u)}   \nu (g,\cdot),
\]
which shows $\cost(\Sh_{\tilde \psi,\tilde \nu,\tilde \varphi}^{\text{int}}) \le \cost(\Sh^{\text{sde}}_{\psi,\nu,\varphi})$ and thus finishes the proof of \eqref{det2}.
\end{proof}

\begin{lemma}\label{pr10}
Let $n\in\N$ and put $z_n = \max(n, u(0))$. Then there exist functions
\[
\hehe_{1,+},\hehe_{1,-},\dots,\hehe_{2n,+},\hehe_{2n,-}\in \Hc_u
\]
 with the following properties.
\end{lemma}
\begin{itemize}
\item[(i)] For every $i=1,\dots,2n$ we have
\[
 \{\hehe_{i,+}\not=0\}\cup\{\hehe_{i,-}\not=0\} = (u^{-1}(z_n)+\tfrac{i-1}{n},u^{-1}(z_n)+\tfrac{i}{n}).
 \]
\item[(ii)] For all $\delta_1,\dots,\delta_{2n}\in \{-,+\}$ we have $\sum_{i=1}^{2n} \hehe_{i,\delta_i}\in \Hc_u$ and
\[
\int_\R\sin\Bigl(\sum_{i=1}^{2n} \hehe_{i,\delta_i}(x)\Bigr)\, dx =  \sum_{i=1}^{2n}  \int_\R\sin(\hehe_{i,\delta_i}(x))\, dx.
\]
\item[(iii)] For every $i=1,\dots,2n$ we have
\[
\int_\R\sin(\hehe_{i,+}(x))\cdot \exp(-\tfrac{x^2}{2})\, dx\geq  \frac{\sin(\tfrac{1}{12})\cdot \exp(-4)}{3}\ \cdot \frac{\exp(-(u^{-1}(z_n))^2)}{n}.
\]
\end{itemize}

\begin{proof}
Let $n\in\N$ and define a function $\hehe\colon \R\to \R $ by
\[
\hehe(x) = \begin{cases}
\thet_{0,\tfrac{1}{3n},0,\tfrac{1}{12} }(x), & \text{if }x\leq \tfrac{2}{3n},\\
\tfrac{1}{12}-\thet_{\tfrac{2}{3n},\tfrac{1}{n},0,\tfrac{1}{12}
}(x),& \text{if }x > \tfrac{2}{3n}.
\end{cases}
\]
By Lemma~\ref{pr1} we have $\hehe\in
C^{\infty}(\R; \R)$ with
\begin{equation}\label{z2}
\{h\neq 0\} = (0,1/n),
\end{equation}
and
\begin{equation}\label{z2a}
0\le h \le \tfrac{1}{12}
\end{equation}
and
\begin{equation}\label{z3}
\|h'\|_\infty \le  n.
\end{equation}

For all $i\in\{1,\dots,2n\}$ and $\delta\in\{-,+\}$ we define
\[
\hehe_{i,\delta}\colon \R\to \R,\quad x\mapsto \delta \hehe\bigl(x-u^{-1}(z_n)-\tfrac{i-1}{n}\bigr).
\]

Clearly, \eqref{z2} implies Property (i).
Moreover, $\hehe_{i,\delta}\in C^{\infty}(\R; \R)$, and by \eqref{z3} and the fact that $u$ is increasing we get for every $x\in (u^{-1}(z_n)+\tfrac{i-1}{n},u^{-1}(z_n)+\tfrac{i}{n})$ that
\begin{equation}\label{z4}
|\hehe_{i,\delta}'(x)| \le n \le z_n = u(u^{-1}(z_n))\leq u(x),
\end{equation}
which proves that $\hehe_{i,\delta}\in \Hc_u$.

Let $\delta_1,\dots,\delta_n\in \{-,+\}$. Then $\sum_{i=1}^{2n} \hehe_{i,\delta_i}\in C^{\infty}(\R; \R)$, and using Property (i) as well as \eqref{z4} we have
$|\sum_{i=1}^{2n} \hehe_{i,\delta_i}'(x)| \le u(x)$
for every $x\in \R$. Thus $\sum_{i=1}^{2n} \hehe_{i,\delta_i}\in \Hc_u$. Furthermore, Property (i) implies that
\[
\sin \circ \sum_{i=1}^{2n} \hehe_{i,\delta_i} = \sum_{i=1}^{2n} \sin \circ \hehe_{i,\delta_i},
\]
which finishes the proof of Property (ii).

Finally, using Property (i) and \eqref{z2a} we get
\begin{align*}
\int_\R\sin(\hehe_{i,+}(x))\cdot \exp(-\tfrac{x^2}{2})\, dx&= \int_{u^{-1}(z_n)+\tfrac{i-1}{n}}^{u^{-1}(z_n)+\tfrac{i}{n}} \sin \bigl(\hehe_{i, +}(x)\bigr) \cdot \exp(-\tfrac{x^2}{2})\,dx\\
&\geq \exp(-\tfrac{(u^{-1}(z_n)+2)^2}{2})\cdot \int_{u^{-1}(z_n)+(i-1)/n}^{u^{-1}(z_n)+i/n} \sin \bigl(\hehe_{i,+}(x)\bigr)\,dx\\
& = \exp(-\tfrac{(u^{-1}(z_n)+2)^2}{2})\cdot \int_0^{\tfrac{1}{n}} \sin(h(x))\, dx\\
& \ge  \exp(-\tfrac{(u^{-1}(z_n)+2)^2}{2})\cdot \int_{\tfrac{1}{3n}}^{\tfrac{2}{3n}} \sin(h(x))\, dx\\
& =  \exp(-\tfrac{(u^{-1}(z_n)+2)^2}{2})\cdot\sin\bigl(\tfrac{1}{12}\bigr)\cdot \frac{1}{3n} \\
& \ge\exp(-(u^{-1}(z_n))^2-4)\cdot
\sin\bigl(\tfrac{1}{12}\bigr)\cdot \frac{1}{3n}
\end{align*}
for all $i=1,\dots,2n$, which shows Property (iii) and completes the proof of the lemma.
\end{proof}

\begin{prop}\label{pr11}
For every $n\in\N$ we have
\[
 e_n^{\text{det}}(\ccF(\Hc_u);S^{\text{int}}) \ge \frac{\sqrt{2}\cdot \sin(\tfrac{1}{12})\cdot \exp(-4)}{6\sqrt{\pi}}\ \cdot \exp(-(u^{-1}(\max(n,u(0))))^2).
\]
\end{prop}

\begin{proof}
Let $n\in\N$. We may apply Proposition~\ref{prop2} with $m=2n$, $b^* = 0 \in \R^{\N_0}$ and $g_{i,\delta} = (h^{(k)}_{i,\delta})_{k\in\N_0}$ for all $i\in\{1,\dots,2n\}$ and $\delta\in\{+,-\}$ to obtain the desired lower bound. Indeed, Lemma~\ref{pr10}(i),(ii) imply that the conditions (i),(ii) in Proposition~\ref{prop2} are satisfied. Furthermore, by Lemma~\ref{pr10}(iii) we obtain
\begin{align*}
S^\text{int}(g_{i,+}) - S^\text{int}(g_{i,-}) & = \frac{1}{\sqrt{2\pi}}\int_\R \bigl(\sin(h_{i,+}(x))-\sin(h_{i,-}(x)) \bigr)\cdot \exp(-\tfrac{x^2}{2})\,dx\\
& = \frac{\sqrt{2}}{\sqrt{\pi}}\int_\R \sin(h_{i,+}(x))\cdot \exp(-\tfrac{x^2}{2})\,dx\\
& \ge \frac{\sqrt{2}}{\sqrt{\pi}}\cdot \frac{\sin(\tfrac{1}{12})\cdot \exp(-4)}{3}\ \cdot \frac{\exp(-(u^{-1}(\max(n,u(0))))^2)}{n}
\end{align*}
for all $i\in\{1,\dots,2n\}$, which completes the proof.
\end{proof}

Clearly, Lemma~\ref{pr9} and Proposition~\ref{pr11} jointly imply Theorem \ref{mainThm2}.

\section*{Acknowledgement}
We are
grateful to Stefan Heinrich, Arnulf Jentzen and Erich Novak for stimulating discussions
on the topic of this paper.

\bibliographystyle{acm}
\bibliography{bibfile}

\end{document}